\documentclass[11pt]{article}
\usepackage[utf8]{inputenc}
\usepackage[T1]{fontenc}
\usepackage[margin=1.1in]{geometry}
\usepackage{amsmath}
\usepackage{amssymb}
\usepackage{graphicx}
\usepackage{url}
\usepackage{xcolor}
\usepackage{amsthm}
\usepackage{caption}
\usepackage{subcaption}
\usepackage{algpseudocode}
\usepackage{algorithm}
\usepackage[section]{placeins}
\usepackage{hyperref}

\newtheorem{theorem}{Theorem}[section]
\newtheorem{lemma}[theorem]{Lemma}
\newtheorem{proposition}[theorem]{Proposition}
\newtheorem{corollary}[theorem]{Corollary}
\newtheorem{definition}[theorem]{Definition}
\newtheorem{assumptions}[theorem]{Assumptions}
\newtheorem{example}[theorem]{Example}
\newtheorem{remark}[theorem]{Remark}

\begin{document}
	
%\title{Computing the free convolution of measures exploiting the trapezoidal quadrature rule}

\title{Computing Free Convolutions via Contour Integrals}

\author{Alice Cortinovis\footnote{Department of Mathematics, Stanford University, CA, USA. Email: \texttt{alicecor@stanford.edu}} \and Lexing Ying\footnote{Department of Mathematics and ICME, Stanford University, CA, USA. Email: \texttt{lexing@stanford.edu}}}

\date{}
\maketitle

\begin{abstract}
This work proposes algorithms for computing additive and multiplicative free convolutions of two given measures. We consider measures with compact support whose free convolution results in a measure with a density function that exhibits a square-root decay at the boundary (for example, the semicircle distribution or the Marchenko-Pastur distribution). A key ingredient of our method is rewriting the intermediate quantities of the free convolution using the Cauchy integral formula and then discretizing these integrals using the trapezoidal quadrature rule, which converges exponentially fast under suitable analyticity properties of the functions to be integrated. 
\end{abstract}	

%=======================================================================================	
\section{Introduction}

% MOTIVATION
 % \paragraph{Motivation.}
Given two independent commutative random variables $X$ and $Y$, the density of the sum can be obtained by computing the convolution of the densities of $X$ and $Y$. This no longer holds when $X$ and $Y$ do not commute. As an example, consider two symmetric matrices $X$ and $Y$ of size $n \times n$ whose eigenvalues are $\pm 1$, each chosen with probability $\frac{1}{2}$, and whose eigenvector matrices are two independently chosen random orthogonal matrix (that is, from the Haar distribution). When the size $n$ grows to infinity, the eigenvalues of $X + Y$ tend to follow an arcsine distribution -- a continuous distribution that does not correspond to the convolution of the random variables that define the eigenvalues of $X$ and $Y$. The theory of free probability, introduced by Voiculescu in the 1980s, is a powerful tool to explain the behavior of the eigenvalue distribution of the sum of random matrices via the concept of \emph{free additive convolution}; we refer the reader to the book~\cite{Mingo2017} for an explanation of the topic and its connections with random matrix theory.
 
 Another application of free probability is the following. When $Z$ is a random matrix of size $n \times p$ whose rows are zero-mean random vectors with variance $\Sigma_p \in \mathbb{R}^{p \times p}$, one can consider the sample covariance $\widehat \Sigma_p:= \frac{1}{n} Z^T Z$. If we consider increasing values of $n$ and $p$ with $\frac{p}{n} \to \gamma$, and the eigenvalues of $\Sigma_p$ have a limit distribution, then the eigenvalue distribution of $\widehat \Sigma_p$ also has a deterministic limit, under mild conditions~\cite{Silverstein1995,Choi1995}. Again, such a limit can be expressed using the tools of \emph{free multiplicative convolution}.

 % RELATED LITERATURE
 \paragraph{Related literature.} The free additive and multiplicative convolutions of measures are defined, theoretically, via Cauchy transforms, R-transforms, T-transforms, and S-transforms, which will be recalled in Section~\ref{sec:freeconvolution}. It is possible to compute the free convolution analytically only in some very special cases; for all other measures, one must resort to numerical techniques. An approach for the numerical computation of free convolutions is to use fixed-point iterations, for example, based on the results in~\cite[Theorem 6.5]{Helton2018}. The work~\cite{Dobriban2015} proposes an algorithm for computing the free multiplicative convolution of a sum of point measures with a Marchenko-Pastur distribution (which fits into the application of free multiplicative convolution mentioned above). The work~\cite{Rao2008} addresses the free additive convolution of measures whose Cauchy transform is an algebraic function. The work~\cite{Olver2012} proposes an algorithm for the free convolution of measures that result in a smoothly-decaying measure supported on the entire real line or a measure that is supported on a compact interval and exhibits a square-root behavior at the boundary; see Remark~\ref{rmk:compare} for a comparison with our approach.

% OUR CONTRIBUTIONS
\paragraph{Contributions.} In this work, we propose a new algorithm for computing the free (additive and multiplicative) convolution of two measures $\mu_1$ and $\mu_2$. We assume that $\mu_1$ and $\mu_2$ have compact support and that their free convolution has a square-root behavior at the boundary, which is made precise in Definition~\ref{def:sqrtbehavior}. Our approach is based on approximating the analytic transforms needed in the computation of the free convolution using the Cauchy integral theorem along suitably defined curves. The numerical evaluation of integrals is done using the trapezoidal quadrature rule, which results in a fast convergence with respect to the number of quadrature points because the functions to be integrated are analytic.

 % OUTLINE
\paragraph{Outline.} The rest of the paper is organized as follows. In Section~\ref{sec:freeconvolution}, we recall the relevant definitions and results on the free convolution of measures. In Section~\ref{sec:algorithmsum}, we outline the proposed algorithm for free additive convolution. A similar algorithm for the free multiplicative convolution is summarized in Section~\ref{sec:algorithmtimes}. In Section~\ref{sec:erroranalysis}, we discuss the approximation errors in the algorithms for free  convolutions. %\LY{mention that the error bounds for the multiplicative case are the same in that section.} 
We report various numerical examples in Section~\ref{sec:experiments}, and the conclusions are given in Section~\ref{sec:conclusions}.	

%We do not provide a discussion on the error bounds for the multiplicative case because it is analogous to the additive case.

%=======================================================================================	
\section{Preliminaries on free convolution}\label{sec:freeconvolution}

In this work, we consider measures with compact support and with a continuous density without atoms.

\paragraph{Notation.} We denote by $\mu$ a measure with support in $[a, b]$ and density $f(x)$. The input measures are $\mu_1$ (with support $[a_1, b_1]$) and $\mu_2$ (with support $[a_2, b_2]$). We denote by $\mu_1 \boxplus \mu_2$ their free additive convolution and by $\mu_1 \boxtimes \mu_2$ their free multiplicative convolution.
We denote $\mathbb{D}$ the open unit disk in $\mathbb{C}$ and $\partial \mathbb{D}$ the unit circle (the border of the unit disk) in $\mathbb{C}$. When computing integrals on $\partial\mathbb{D}$ or any other curve in the complex plane, these are intended in an anti-clockwise direction. $\mathbb{C}^+$ and $\mathbb{C}^{-}$ denote the half of the complex plane containing numbers with strictly positive or negative imaginary parts, respectively.

\subsection{Measures with sqrt-behavior at the boundary and Jacobi measures}

The most regular distributions we consider in this paper are those with an sqrt-behavior at the boundary in the following sense.

\begin{definition}\label{def:sqrtbehavior}
	The measure $\mu$ has square-root behavior at the boundary (\emph{sqrt-behavior}) if its density has the form
	\begin{equation*}
		\mathrm{d}\mu (x) = f(x) \mathrm{d}x = \psi(x)\sqrt{x-a}\sqrt{b-x} \mathrm{d}x
	\end{equation*}
for some $\psi \in C^1([a,b])$ with $\psi'(x)$ of bounded variation. 
\end{definition}

We recall two examples of measures of this type that arise naturally from random matrix theory.

\begin{example}[Semicircle]\label{ex:S}
  The \emph{semicircle} distribution has support in $[a, b] = [-2, 2]$ and its density is $$\mathrm{d}\mu(x) = \frac{1}{2\pi}\sqrt{4 - x^2} \mathrm{d}x.$$ It has sqrt-behavior at the boundary with $\psi(x) = \frac{1}{2\pi}$. 
  
  The semicircle distribution is the limit (for $n \to \infty$) eigenvalue distribution of a random symmetric matrix of size $n \times n$ that has random i.i.d. entries with mean $0$ and variance $\frac{1}{n}$ above the diagonal, and i.i.d. entries with mean $0$ and variance $\frac{C}{n}$ on the diagonal, for some constant $C \ge 0$. 	
 	\end{example}
 	
 	\begin{example}[Marchenko-Pastur]\label{ex:MP}
 	 The \emph{Marchenko-Pastur} distribution with parameter $\lambda \in (0,1)$ has support in $[\lambda_-, \lambda_+]$ with $\lambda_{\pm} := (1 \pm \sqrt\lambda)^2$, and its density is $$\mathrm{d}\mu(x) = \frac{\sqrt{(\lambda_+ - x)(x - \lambda_-)}}{2\pi\lambda x} \mathrm{d}x.$$
 	It has sqrt-behavior at the boundary, with $\psi(x) = \frac{1}{2\pi \lambda x}$. 
  
  The Marchenko-Pastur distribution is the limit (for $n \to \infty$) eigenvalue distribution of random matrices of the form $B_nB_n^T$ for matrices $B_n \in \mathbb{R}^{\lfloor \lambda n \rfloor \times n}$ made of i.i.d. entries with zero mean and variance $\frac{1}{n}$.
 	\end{example}

  A more general class of measures is the following.

  \begin{definition}
      $\mu$ is a Jacobi measure if its density has the form
      \begin{equation*}
          \mathrm{d}\mu(x) = (x-a)^\alpha (b - x)^\beta \psi(x) \mathrm{d}x 
      \end{equation*}
      for some $\alpha,\beta >-1$ and $\psi \in C^1([a, b])$.
  \end{definition}

  \begin{example}[Uniform distribution]
 	As an example of a Jacobi measure that does \emph{not} have an sqrt-behavior at the boundary (that we will use in our numerical experiments), let us consider the \emph{uniform} distribution with support in the interval $[a, b]$, for some $a, b\in \mathbb{R}$. The density is $\mathrm{d}\mu(x) = \frac{1}{b-a}\mathrm{d}x$. 
 	\end{example}

\subsection{The Cauchy transform and Stiltjies inversion}

\begin{definition}
	The Cauchy transform of the point $z \in \mathbb{C} \backslash [a, b]$ is defined as
	\begin{equation*}%\label{eq:defcauchy}
		G(z) = \int_{a}^{b} \frac{1}{z-x} \mathrm{d}\mu(x).
	\end{equation*}
\end{definition}
When ambiguous, we will put a subscript indicating the measure, for example, $G_{\mu}(z)$.

\begin{theorem}\label{thm:Ganalytic}
	The Cauchy transform $G$ is analytic on $\mathbb{C} \backslash [a,b]$ and at infinity.
\end{theorem}
This is a version of Proposition A.7 in~\cite{Olver2012} with different assumptions, and it is slightly better than Lemma 2 in Chapter 3 in~\cite{Mingo2017} because we have stronger assumptions in our case (i.e. compact support). The proof follows~\cite[Chapter 3, Lemma 2]{Mingo2017}; we report the proof in Appendix~\ref{sec:proofs} for completeness.

To see the Cauchy transform as a map from the unit disk in $\mathbb{C}$ to $\mathbb{C}$, we use the Joukowski map.

\begin{definition}
	The Joukowski map relative to the segment $[a, b]$ is
	\begin{equation*}
		J_{[a,b]}(v) = \frac{1}{2} \left ( v + \frac{1}{v} \right ) \frac{b-a}{2} + \frac{b+a}{2}, \qquad v \in \mathbb{C}.
	\end{equation*}
\end{definition}
This is a conformal map from the unit disk to $\mathbb{C} \cup \{\infty\}\backslash [a, b]$. The unit circle is mapped into the segment $[a, b]$ (``twice''), and circles are mapped into ellipses. For $z \in \mathbb{C} \cup \{\infty\}\backslash [a, b]$ we denote by $J_{[a,b]}^{-}$ and $J_{[a,b]}^{+}$ the inverse Joukowski functions which are inside and outside the unit disk, respectively. The Joukowski map allows us to define
\begin{equation*}
	\mathcal G(v) := G\left (J_{[a,b]}(v)\right ),
\end{equation*}
%\LY{I believe tilde is used to denote numerical approximation later. So maybe we should not use a tilde in the above notation. Maybe $\mathsf{G}$}
which is a holomorphic function that maps the unit disk into $\mathbb{C}$, and such that $\mathcal G(0) = 0$. Therefore, we can write it as a power series
\begin{equation}\label{eq:laurent}
	\mathcal G(v) = \sum_{n \ge 1} g_n v^n
\end{equation}
that converges everywhere inside $\mathbb{D}$.

\begin{remark}\label{rmk:convergenceGtilde}
    If the measure $\mu$ has sqrt-behavior at the boundary, the series~\eqref{eq:laurent} converges also \emph{on} the unit circle. Indeed, as discussed in~\cite[Section 3.0.1]{Olver2012}, the coefficients are $g_n = \frac{(b-a)\pi}{2} \psi_{n-1}$, where $\psi_{n}$ is the $n$-th coefficient in the series expansion of $\psi(x)$ with Chebyshev polynomials of the second kind. When $\psi(x)$ is analytic in a (complex) neighborhood of the interval $[a,b]$, the coefficients $\psi_n$ decay exponentially fast; see, e.g.,~\cite[Theorem 8.1]{Trefethen2019}. To obtain convergence of~\eqref{eq:laurent} on the unit circle, it is sufficient to assume that $\psi(x)$ is Lipschitz-continuous, see, e.g.,~\cite[Theorem 3.1]{Trefethen2019}.
\end{remark}

\begin{example}
For the semicircle and Marchenko-Pastur distribution, we can write the Cauchy transform explicitly (and their inverse and $\mathcal G(z)$ as well). For the semicircle distribution, we have
\begin{equation*}
    G(z) = \frac{z - \sqrt{z^2-4}}{2}, \qquad G^{-1}(w) = w + \frac{1}{w},\qquad \mathcal G(v) =v.
\end{equation*}
For the Marchenko-Pastur distribution with parameter $\lambda$, we have
 	\begin{equation*}
 		G(z) = \frac{z+\lambda-1-\sqrt{(z-\lambda-1)^2-4\lambda}}{2\lambda z}, \qquad G^{-1}(w) = \frac{1}{1-\lambda w} + \frac{1}{w}, \qquad \mathcal G(v) = \frac{v}{\lambda(1-v\sqrt{\lambda})}.
 	\end{equation*}
 	In particular, note that $\mathcal G$ is analytically continuable in the disk of radius $\frac{1}{\sqrt\lambda} > 1$.
  For the uniform distribution on the interval $[-m,m]$, instead, we have that
 	\begin{equation*}
 		G(z) = \frac{1}{2m} \log \frac{z+m}{z-m}, \qquad G^{-1}(w) = -m + \frac{2m}{1-\exp(-2mw)}.
 	\end{equation*}
 	We have
 	\begin{equation*}
 		\mathcal G(v) = \frac{1}{m}\log \left ( 1 + \frac{2}{1-v}\right ) = \frac{2}{m} \sum_{n \ge 0} \frac{v^{2n+1}}{2n+1},
 	\end{equation*}
 	which means that $\mathcal G$ is not analytically continuable on the border of the unit disk.
\end{example}

The Stiltjies inversion formula allows us to recover the density of a distribution from its Cauchy transform.

\begin{theorem}[{\cite[Theorem 6]{Mingo2017}}]\label{thm:stiltjies}
	Assume $\mathrm{d}\mu(x) = f(x)\mathrm{d}x$ for a continuous function $f$. For any $c < d \in (a, b)$ we have that
	\begin{equation*}
		-\frac{1}{\pi} \lim_{y \to 0^+} \int_c^d \mathrm{Im}\left ( G(x + iy) \right ) \mathrm{d}x = \mu([c, d]).
	\end{equation*}
\end{theorem}
In our algorithm, we will use the following corollary, whose proof we report in Appendix~\ref{sec:proofs}, for completeness.

\begin{corollary}\label{cor:limit}
	Let $\mathrm{d}\mu(x) = f(x)\mathrm{d}x$ for a continuous function $f$. Then for all $\theta \in (0, \pi)$ we have
	\begin{equation*}
		f\left ( \frac{b-a}{2}\cos\theta + \frac{b+a}{2}\right ) = \frac{1}{\pi} \lim_{r \to 1^-} \mathrm{Im} \left ( \mathcal G(r e^{i\theta})\right ).
	\end{equation*}
\end{corollary}

\subsection{The free additive convolution}

The key function needed to define the free additive convolution of measures is the R-transform, which is defined as follows.

\begin{definition}
	The R-transform of a measure $\mu$ is a function $R$ defined in a neighborhood of $0$ such that $G \left ( R(w) + \frac{1}{w} \right)=w $. 
\end{definition}
When the Cauchy transform $G$ is invertible, we have $R(w) = G^{-1}(w) - \frac{1}{w}$. When we need to avoid confusion, we will denote the R-transform of the measure $\mu$ by $R_{\mu}(w)$. 

\begin{theorem}\label{thm:Ranalytic}
	Let $\mu$ be a measure with compact support in $[a, b]$. Then its R-transform is analytic on $G(\mathbb{R}\backslash[a,b])$ and on a disk centered in the origin with radius $\frac{1}{6\min\{|a|, |b|\}}$. 
\end{theorem}

If $G'(z) \neq 0$ for all $z \in \mathbb{C} \cup \{\infty\} \backslash [a,b]$ then $R$ is invertible on $G(\mathbb{C} \cup \{\infty\} \backslash [a,b])$ and it is analytic on the whole $G(\mathbb{C} \cup \{\infty\} \backslash [a,b])$. 
We have the following result for the free additive convolution of $\mu_1$ and $\mu_2$.

\begin{theorem}[{\cite[Chapter 2, Theorem 18]{Mingo2017}}]\label{thm:sum}
	Given two measures $\mu_1$ and $\mu_2$, 
	\begin{equation*}
		R_{\mu_1 \boxplus \mu_2}(w) = R_{\mu_1}(w) + R_{\mu_2}(w).
	\end{equation*}
\end{theorem}

\subsection{The free multiplicative convolution}

For two measures $\mu_1$ and $\mu_2$ with compact support, analogously to the additive case, a suitable transform (S-transform) allows us to compute the free multiplicative convolutions.

\begin{definition}
	The T-transform and S-transform of a measure $\mu$ are
 \begin{equation*}
 T(z) = \int_a^b \frac{t}{z-t} \mathrm{d}\mu(t) \quad \text{ and } \quad  S(w) = \frac{1+w}{w T^{-1}(w)},
 \end{equation*}
 respectively.
\end{definition}

The T-transform maps $\mathbb{C}^{+}$ into $\mathbb{C}^{-}$ and maps $\mathbb{C}^{-}$ into $\mathbb{C}^{+}$, and it is a holomorphic function from $\mathbb{C} \cup \{\infty\} \backslash [a, b]$ to $\mathbb{C}$. Therefore, the map $\mathcal T(v):= T(J(v))$ has a power expansion $\mathcal T(v) = \sum_{n \ge 1} t_n v^n$ that converges inside the unit disk. 
The S-transform is analytic in a neighborhood of $0$. We skip the details because this is similar to the additive case.

\begin{theorem}[{\cite[Chapter 4, Theorem 23]{Mingo2017}}]
	Given two measures $\mu_1$ and $\mu_2$,
	\begin{equation*}
		S_{\mu_1 \boxtimes \mu_2}(w) = S_{\mu_1}(w) \cdot S_{\mu_2}(w).
	\end{equation*}
\end{theorem}

%=======================================================================================	
\section{An algorithm for free additive convolution}\label{sec:algorithmsum}

In this section, we propose an efficient algorithm for computing the free additive convolution of two measures $\mu_1$ and $\mu_2$ that satisfy the following set of assumptions. 

\begin{assumptions}\label{ass:sum}
We assume that $\mu_1$ and $\mu_2$ satisfy the following properties.
\begin{itemize}
    \item They have compact support $[a_1, b_1]$ and $[a_2, b_2]$.
    \item One of them, without loss of generality $\mu_1$, has sqrt-behavior at the boundary.
    \item The other one, without loss of generality $\mu_2$, is a Jacobi measure.
    \item The Cauchy transforms $G_{\mu_1}$ and $G_{\mu_2}$ are invertible on their domain of definition.
\end{itemize}
\end{assumptions}

 We recall that the Cauchy integral formula states that, for an analytic function $f$ and a contour $\Gamma$, for any point $z$ inside $\Gamma$ 
 \begin{equation*}
     f(z) = \frac{1}{2\pi i}\int_{\Gamma} \frac{f(w)}{w-z} \mathrm{d}w.
 \end{equation*}
 This is a tool that we will use frequently in our algorithm, to evaluate Cauchy transforms and R-transforms (which are analytic functions).
 
Another central theoretical result for our algorithm is the following theorem that allows us to characterize the support of $\mu_1 \boxplus \mu_2$ and the image of $G_{\mu_1 \boxplus \mu_2}$.
\begin{theorem}[{Theorem 2.2 and Theorem 5.2 in~\cite{Olver2012}}]\label{thm:support}
		Under the Assumptions~\ref{ass:sum}, let
	\begin{equation}\label{eq:inverse}
		g(w) := G_{\mu_1}^{-1}(w) + G_{\mu_2}^{-1}(w) - \frac{1}{w},
	\end{equation}
 which coincides with $G_\mu^{-1}(w)$ in a neighborhood of $0$. We have the following properties. %\LY{is $g(w)=G_\mu^{-1}(w)$? shall we say this? AC: added it}
		\begin{itemize}
			\item $\mu:= \mu_1\boxplus \mu_2$ has sqrt-behavior at the boundary.
			\item The support of $\mu$ is contained in the interval 
			\begin{equation*}
				[a, b] := [g(\xi_a), g(\xi_b)],
			\end{equation*}
			where $\xi_a$ and $\xi_b$ are the unique zeros of the derivative $g'$ in the intervals
   \begin{equation*}
   (\max(G_{\mu_1}(a_1), G_{\mu_2}(a_2)), 0)\quad \text{ and } \quad (0, \min(G_{\mu_1}(b_1), G_{\mu_2}(b_2))),
   \end{equation*}
   respectively.
			\item To check whether a point $w$ is in the image of $G_{\mu_1 \boxplus \mu_2}$ we can use the following criterion:
			\begin{equation*}
				w \in G_{\mu_1 \boxplus \mu_2}(\mathbb{C} \backslash [a,b]) \Leftrightarrow
				\mathrm{sgn}(\mathrm{Im}(g(w))) = -\mathrm{sgn}(\mathrm{Im}(w)).
			\end{equation*}
		\end{itemize}
	\end{theorem}

Our proposed algorithm consists of four main steps:
 \begin{enumerate}
     \item Set up the computation of the R-transform for $\mu_1$ and $\mu_2$.
     \item Compute the support of $\mu$
     \item Compute the Cauchy transform of $\mu$ on a suitable set of points.
     \item Recover the density of $\mu$ from its Cauchy transform.
 \end{enumerate}
	
\subsection{Step 1: Setting up the computation of the R-transform}\label{sec:evalG}

In this section, we explain how to set up the computation of the R-transform for a measure $\mu$ with support in $[a,b]$. We apply this procedure to $\mu_1$ and $\mu_2$. Let us consider the curve $\Gamma := \mathcal G(r_A\partial\mathbb{D})$, for some $r_A \in (0,1)$. Since the R-transform is analytic, given a point $w$ inside $\Gamma$ we can write $R(w)$ with the Cauchy integral formula:
\begin{equation*}%\label{eq:intR}
	R(w) = \frac{1}{2\pi i} \int_{\Gamma} \frac{R(s)}{w-s} \mathrm{d}s.
\end{equation*}
As $\Gamma$ is parameterized by $s = \mathcal G(r_A v)$ for $v$ on the unit circle, we have $\mathrm{d}s = r_A \mathcal G'(r_A v) \mathrm{d}v$ and $G^{-1}(s) = J_{[a,b]}(r_Av)$ and  the integral becomes
\begin{equation}\label{eq:int2}
	R(w) = \frac{r_A}{2\pi i} \int_{\partial \mathbb{D}} \frac{J_{[a,b]}(r_Av) - 1/\mathcal G(r_Av)}{\mathcal G(r_Av)-w} G'(J_{[a,b]}(r_A v)) J_{[a,b]}'(r_Av)\mathrm{d}v =: \int_{\partial \mathbb D} u(v) \mathrm{d}v.
\end{equation}
 We perform an additional change of variables $v = \exp(2\pi i\theta)$, for $\theta \in [0, 1]$, and aim at discretizing~\eqref{eq:int2} with the trapezoidal quadrature rule, which reads
\begin{equation*}%\label{eq:R1}
	R(w) \approx \frac{r_A}{N} \sum_{j=0}^{N-1} \xi_N^{j} \cdot G'(J_{[a, b]}(\xi_N^j)) \cdot J_{[a, b]}'(\xi_N^j) \cdot \frac{J_{[a, b]}(\xi_N^j) - \frac{1}{G(J_{[a, b]}(\xi_N^j))}}{ G(J_{[a, b]}(\xi_N^j))- w}, \qquad \xi_N = \exp\left ( \frac{2\pi i}{N} \right ).
\end{equation*}
 Note that the quantities $G(J_{[a, b]}(\xi_N^j))$ and $G'(J_{[a, b]}(\xi_N^j))$ do not depend on the point $w$ and can be computed ahead of time. For their computation, we again use the trapezoidal quadrature rule. More precisely, for evaluating the Cauchy transform $G(z)$ for a point  $z \in \mathbb{C} \backslash [a,b]$, we rewrite the integral that defines it as
	\begin{equation*}%\label{eq:G1}
		G(z) = \int_{\gamma}\frac{1}{2}\frac{b-a}{2}\left ( 1 - \frac{1}{v^2}\right ) \frac{f(J_{[a,b]}(v))}{z- J_{[a,b]}(v)} \mathrm{d}v,
	\end{equation*}
	where $\gamma$ is the upper unit semicircle in the complex plane, oriented clockwise, where we made the change of variables $x = J_{[a,b]}(v)$. 	
	With an additional change-of-variable $v = \exp (2\pi i \theta)$ for $\theta \in [0,1]$ we get
	\begin{equation}\label{eq:that}
		G(z) = \pi \frac{b-a}{2}\int_1^0 \sin(2\pi\theta) \frac{f\left ( \frac{b-a}{2}\cos (2\pi\theta) + \frac{b+a}{2} \right )}{z - \left ( \frac{b-a}{2}\cos (2\pi\theta) + \frac{b+a}{2} \right )} \mathrm{d}\theta.
	\end{equation}
The discretization in $N$ quadrature points using the trapezoidal quadrature rule reads:
\begin{equation}\label{eq:quadratureG}
	G(z) \approx - \frac{\pi}{N} \sum_{k=0}^{N}{}^{''} \frac{b-a}{2} \sin\left ( \frac{2\pi k}{N}\right ) \frac{f\left ( \frac{b-a}{2}\cos\left ( \frac{2\pi k}{N}\right ) + \frac{b+a}{2} \right )}{z - \left ( \frac{b-a}{2}\cos\left ( \frac{2\pi k}{N}\right ) + \frac{b+a}{2} \right )},
\end{equation}
where the ${}^{''}$ denotes that the first and last terms are halved. This makes sense for any measure $\mu$, but it has a provably exponential convergence when $\mu$ has sqrt-behavior at the boundary (see Corollary~\ref{cor:sqrtexp} below). 
Now note that the derivative of the Cauchy transform can be written as
\begin{equation*}
	G'(z) = -\int_a^b \frac{1}{(z-x)^2} \mathrm{d}\mu(x)
\end{equation*}
and this integral can be discretized in the same way as the Cauchy transform itself.

Practically speaking, the first step of our proposed algorithm consists in choosing a radius $r_A < 1$ (for example, $r_A = 0.95$), a number of quadrature points $N$ (for example, $N = 1000$), and approximating 
\begin{equation*}
	c_j^{(k)} \approx G_{\mu_k}(J_{[a_k, b_k]}(\xi_N^j)), \qquad d_j^{(k)} \approx G_{\mu_k}'(J_{[a_k, b_k]}(\xi_N^j))
\end{equation*}
for $j=0,1,\ldots,N$ and for $k=1,2$ via~\eqref{eq:quadratureG}. We call $\Gamma_1$ and $\Gamma_2$ the images of the circle of radius $r_A$ by $\mathcal G_{\mu_1}$ and $\mathcal G_{\mu_2}$, respectively. This allows us to compute the R-transform for $\mu_1$ and $\mu_2$, when needed, from the expression
\begin{equation}\label{eq:R1}
	R_{\mu_k}(w) \approx \frac{r_A}{N} \sum_{j=0}^{N-1} \xi_N^{j} \cdot d_j^{(k)} \cdot J_{[a_k, b_k]}'(\xi_N^j) \cdot \frac{J_{[a_k, b_k]}(\xi_N^j) - 1/c_j^{(k)}}{ c_j^{(k)}- w}, \quad k \in \{1,2\}.
\end{equation}
%\end{footnotesize}

%\LY{shall we say left, middle, right, instead of (1,1), (1,2), and (1,3) blocks? For all the figures}
\begin{example}\label{ex:explain1}
    Throughout this section, we will illustrate the behavior of our proposed algorithm on an explicit example: we consider $\mu_1$ to be a Marchenko-Pastur distribution with parameter $\lambda = 0.5$ and $\mu_2$ to be the semicircle distribution.
    In Figure~\ref{fig:explain1a}, we show, in the left block, a \textcolor{blue}{blue} circle of radius $r_A = 0.95$, in the middle block the image of this circle via the Joukowski transform associated with $\mu_1$, and in the right block the image of this curve (an ellipse) via the Cauchy transform $G_{\mu_1}$. The \textcolor{red}{red} circle in the left plot is the unit circle. In Figure \ref{fig:explain1b}, we do the same for $\mu_2$. The number of discretization points used in the expression~\eqref{eq:quadratureG} is $N = 400$ in these examples.
\end{example}

\begin{figure}[hbt]
    \centering\includegraphics[width=\textwidth]{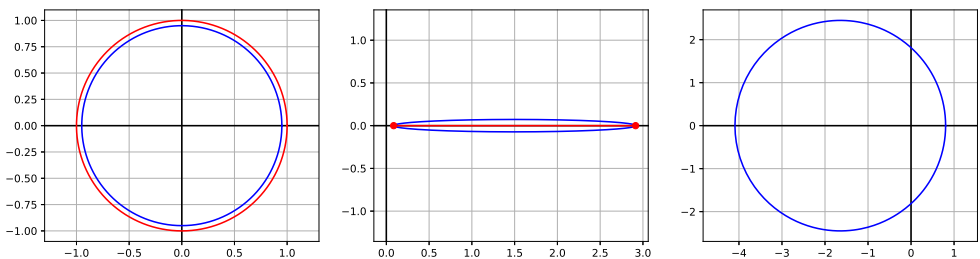}
    \caption{Computation of the Cauchy transform of an ellipse corresponding to a circle of radius $r_A = 0.95$ for the Marchenko-Pastur distribution with parameter $\lambda = 0.5$ (see Example~\ref{ex:explain1}).}
    \label{fig:explain1a}
\end{figure}
\begin{figure}[hbt]
    \centering\includegraphics[width=\textwidth]{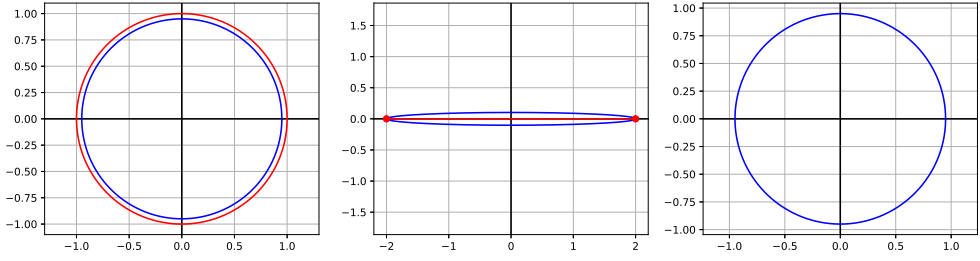}
    \caption{Computation of the Cauchy transform of an ellipse corresponding to a circle of radius $r_A = 0.95$ for the semicircle distribution (see Example~\ref{ex:explain1}).}
    \label{fig:explain1b}
\end{figure}

		\subsection{Step 2: Computing the support of the free sum}\label{sec:support}
	The second point in Theorem~\ref{thm:support} gives a clear characterization of the support of $\mu = \mu_1 \boxplus \mu_2$ we use a result from~\cite{Olver2012}. 
Note that
\begin{equation*}
	g'(w) =  \frac{1}{G_{\mu_1}'(G_{\mu_1}^{-1}(w))}+ \frac{1}{G_{\mu_2}'(G_{\mu_2}^{-1}(w))} + \frac{1}{w^2}
\end{equation*}
so it can be approximated numerically using~\eqref{eq:R1} and~\eqref{eq:quadratureG} when $w$ is inside both $\Gamma_1$ and $\Gamma_2$. In our algorithm, we use bisection to numerically find a zero of $g'(w)$ in the interval $((1-\varepsilon)\max\{c_{N/2}^{(1)}, c_{N/2}^{(2)}\}, 0)$ and one in the interval $(0, (1-\varepsilon)\min\{c_{0}^{(1)}, c_{0}^{(2)}\})$, where the parameter $\varepsilon$ is needed because when $w$ is very close to either $\Gamma_1$ or $\Gamma_2$ the computations become unstable; see Section~\ref{sec:erroranalysis} for a better insight on this problem.
We stop the binary search when the size of the interval is smaller than a fixed tolerance. We denote by $\xi_a$ and $\xi_b$ the zeros that we found numerically. Then, we compute $a \approx g(\xi_a)$ and $b  \approx g(\xi_b)$, and these are (up to numerical errors) the extrema of the support of $\mu = \mu_1 \boxplus \mu_2$.
	
\subsection{Step 3: Computation of the Cauchy transform of $\mu$}\label{sec:integralG}
	
The next step in our algorithm is finding a circle of radius $r_B$ that is entirely contained in $G_{\mu}(\mathbb{C} \cup \{\infty\} \backslash [a,b])$. We want $r_B$ to be as large as possible, for reasons that will be clarified in Section~\ref{sec:erroranalysis}. An upper bound for $r_B$ is the radius of the largest circle that is contained inside the intersection of $\Gamma_1$ and $\Gamma_2$. Given a point $w$ inside this intersection, the criterion from Theorem~\ref{thm:support} allows us to check whether $w$ is in the image of $G_{\mu}$ or not. Therefore, we can perform a binary search on the radius and find a suitable value of $r_B$. 
	
Consider the curve $\Gamma := g(r_B \partial \mathbb{D}) = G_{\mu}^{-1}(r_B \partial \mathbb{D})$. Leveraging the analyticity of $G_{\mu}$, for a point $z$ inside $\Gamma$ we can write $G_{\mu}(z)$ using the Cauchy integral formula
\begin{equation*}%\label{eq:cauchyoutside}
		G_{\mu}(z) = \frac{1}{2\pi i} \int_{\Gamma} \frac{ G_{\mu}(s)}{s - z}\mathrm{d}s = \frac{r_B^2}{2\pi i} \int_{\partial \mathbb{D}} \frac{g'(r_B v) v}{g(r_B v) - z}  \mathrm{d}v.
	\end{equation*}
where the second equality follows from the parametrization of $\Gamma$ with $s =  G_{\mu}^{-1}(r_Bv)$ for  $v \in \partial \mathbb{D}$. The trapezoidal quadrature rule for this integral gives
	\begin{equation}\label{eq:int4}
		G_{\mu}(z) \approx \frac{r_B^2}{N} \sum_{j=0}^{N-1} \frac{g'(r_B \xi_N^j) \xi_N^{2j}}{g(r_B \xi_N^j) - z}.
	\end{equation}
We now choose a circle of radius $r_C$ which is contained inside $J_{[a,b]}^{-}(\Gamma)$, and we evaluate $\mathcal G(v) = G(J_{[a,b]}(v))$ numerically on the points $r_C \xi_M^j$, for $j=0,1,\ldots,M$. Let us denote by $h_0, h_1, \ldots, h_M$ the values obtained by the quadrature formula~\eqref{eq:int4}. We will use these to recover an approximation of the measure $\mu$ in the next section.

\begin{example}\label{ex:explain2}
    In Figure~\ref{fig:explain2}, we show, in the middle block, the numerically computed extrema $a$ and $b$ (the red dots) of the support of $\mu = \mu_1 \boxplus \mu_2$, where $\mu_1$ and $\mu_2$ are the Marchenko-Pastur distribution with parameter $\lambda = 0.5$ and the semicircle distribution from Example~\ref{ex:explain1}, respectively. The circle of radius $r_B$ is drawn in \textcolor{orange}{orange} in the right block. The curve $\Gamma$ is the \textcolor{orange}{orange} curve in the middle block, and the \textcolor{orange}{orange} curve in the left block is $J_{[a,b]}^{-}(\Gamma)$. The circle of radius $r_C$ is in \textcolor{green}{green} in the left block, its image via $J_{[a,b]}$ is the \textcolor{green}{green} ellipse in the middle block and its image via $\mathcal G_{\mu}$ is the \textcolor{green}{green} curve in the right block.
\end{example}

\begin{figure}[hbt]
    \centering\includegraphics[width=\textwidth]{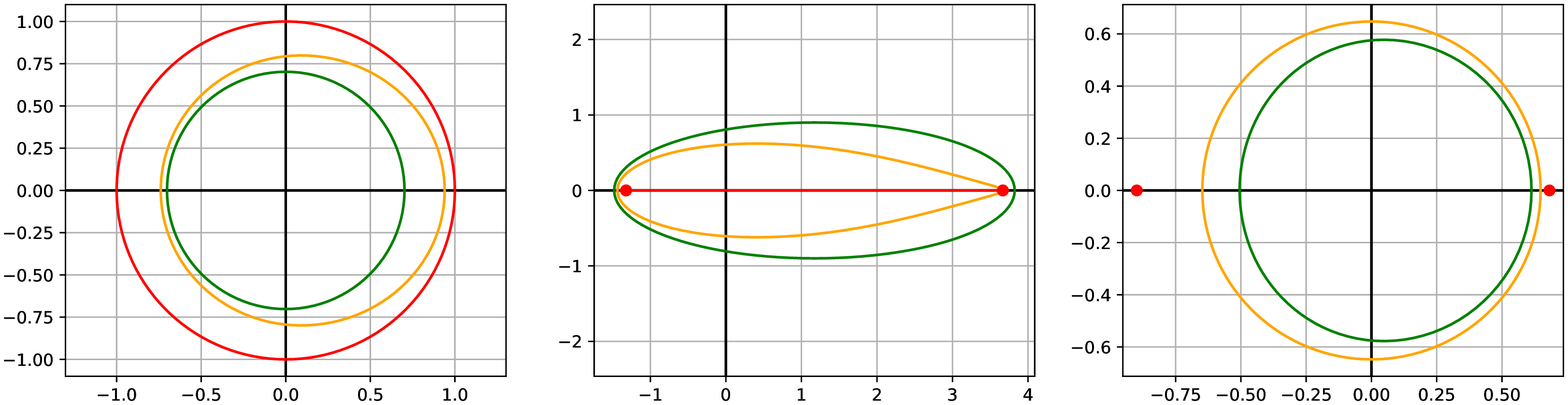}
    \caption{Illustration of Step 3 of the proposed algorithm for free additive convolution, see Example~\ref{ex:explain2}.}
    \label{fig:explain2}
\end{figure}

\subsection{Step 4: Recovering the density of $\mu$ from its Cauchy transform}\label{sec:series}

From the previous step of the algorithm, we have the (approximate) value of $\mathcal G_{\mu}$ in $M$ equispaced points on a circle of radius $r_C < 1$. If we truncate the power series~\eqref{eq:laurent} to the first $M$ terms, we obtain the relation
\begin{equation}\label{eq:IFFT}
	\begin{bmatrix}
		\mathcal G_{\mu}(r_C \xi_M^0)\\ \mathcal G_{\mu}(r_C \xi_M^1) \\ \vdots \\ \mathcal G_{\mu}(r_C \xi_M^{M-1}) 
	\end{bmatrix} = \begin{bmatrix}
		\xi_M^0 & & & \\
		& \xi_M^1 & & \\
		& & \ddots & \\
		& & & \xi_M^{M-1}	
	\end{bmatrix} \cdot F \cdot \begin{bmatrix} g_1 r_C \\ g_2 r_C^2 \\ \vdots \\ g_{M} r_C^{M} \end{bmatrix},
\end{equation}
where $F$ is the Fourier matrix of size $M \times M$. Therefore, we can recover the (approximate) coefficients $g_1, \ldots, g_{M}$ by performing an IFFT to the vector $[h_0, h_1, \ldots, h_{M-1}]^T$ that approximates the left-hand-side and then dividing the $j$th entry by $r_C^j$ for $j=1,\ldots,M$.
Note that, when $r_C^j$ is very small, we cannot hope that the corresponding coefficient of the power series is numerically accurate. Hence, we truncate the series to the first $m$ terms, for some $m < M$. Now using Theorem~\ref{thm:limit} we take the following approximation for the density $\mathrm{d}\mu(x) = f(x)\mathrm{d}x$:
\begin{equation*}
	f\left ( J_{[a,b]}(\exp(i\theta))\right ) \approx -\frac{1}{\pi} \mathrm{Im} \left ( \sum_{j=1}^m g_j \exp(ij\theta)\right ),
\end{equation*}
which can be written nicely in matrix form:
%\begin{scriptsize}
	\begin{equation}\label{eq:recovery}
		\begin{bmatrix}
			f\left (J_{[a,b]}(\xi_m^0) \right )\\
		f\left (J_{[a,b]}(\xi_m^1) \right )\\
			\vdots\\
			f\left (J_{[a,b]}(\xi_m^{m-1}) \right )
		\end{bmatrix} \approx -\frac{1}{\pi}\mathrm{Im} \left (\begin{bmatrix}
			\xi_m^0 & & & \\
			& \xi_m^1 & & \\
			& & \ddots & \\
			& & & \xi_m^{m-1}	
		\end{bmatrix} \cdot F \cdot \begin{bmatrix} g_1\\g_2\\ \vdots \\ g_m \end{bmatrix}  \right ).
	\end{equation}
%\end{scriptsize}
By padding the vector $\begin{bmatrix} g_1 & g_2 & \cdots & g_m\end{bmatrix}^T$ with zeros, we get a discretization of $\mu$ in any number of points. 

We remark that~\eqref{eq:recovery} only makes sense, theoretically, when the power series corresponding to $\mathcal G_{\mu}$ also converges on the unit circle. This is the case when $\mu$ has sqrt-behavior at the boundary, and its density is analytic when non-zero. While Theorem~\ref{thm:support} only states that $\mu = \mu_1 \boxplus \mu_2$ has sqrt-behavior at the boundary, the result from~\cite[Theorem 2.2 and Lemma 2.7]{Bao2020} states that under mild assumptions $\mu$ has the analyticity property as well.

\begin{example}\label{ex:explain3}
    In Figure~\ref{fig:explain3} we plot, in the right block, the approximation of the density of $\mu_1 \boxplus \mu_2$ for the input measures from Examples~\ref{ex:explain1} and~\ref{ex:explain2}, together with the densities of the input measures. In the left block, we compare the obtained approximation of $\mathrm{d}\mu(x)$ with the empirical eigenvalue distribution of a $5000 \times 5000$ matrix resulting from the sum of two random matrices sampled accordingly to Example~\ref{ex:S} and~\ref{ex:MP}.
\end{example}

\begin{figure}[htb]
    \centering\includegraphics[width=.8\textwidth]{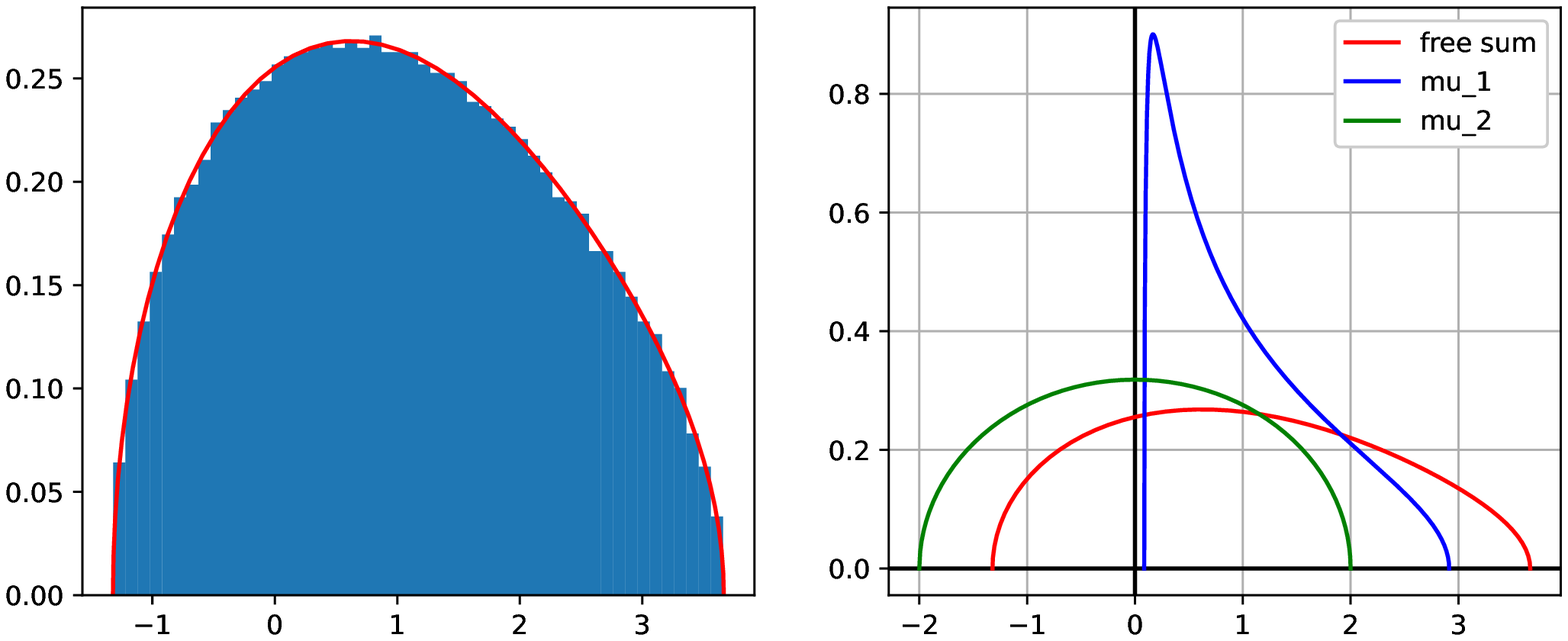}
    \caption{Illustration of Step 4 of the proposed algorithm for free additive convolution, see Example~\ref{ex:explain3}.}
    \label{fig:explain3}
\end{figure}

\begin{algorithm}[htb]
 \begin{small}
		\caption{Computation of the free sum of two measures}
  \label{alg:sum}
		\begin{algorithmic}[1]
			\Require{Measures $\mu_1$ and $\mu_2$, number of quadrature points $N$, number of Fourier coefficients $m$, tolerance $\varepsilon$, radius $r_A$, tolerance \texttt{tol}. Optional: Cauchy transform, inverse Cauchy transform, the derivative of Cauchy transform, for one or both measures.}
			\Ensure{Discretization of the density of the measure $\mu := \mu_1 \boxplus \mu_2$.}
            \State{\% \emph{Step 1: Set up the computation of the R-transform for both input measures}}
			\For{$k=1,2$}
				\If{(Inverse) Cauchy transform of $\mu_k$ is unknown}
					\State{Compute values of $d_j^{(k)}, J_{[a_k, b_k]}'(\xi_N^j), c_j^{(k)}$ appearing in Eq.~\eqref{eq:R1} using~\eqref{eq:quadratureG}.}
				\EndIf
			\EndFor
            \State{\% \emph{Step 2: Find support of $\mu$}}
			\State{Find $\xi_a \in \left (\max\{\mathcal G_{\mu_1}(-r_A+\varepsilon), \mathcal G_{\mu_2}(-r_A+\varepsilon)\}, 0\right )$ such that $g'(\xi_a) = 0$ ($g(z)$ from~\eqref{eq:inverse}).}\label{line:support1}
			\State{Find $\xi_b \in \left (0, \min\{\mathcal G_{\mu_1}(r_A-\varepsilon), \mathcal G_{\mu_2}(r_A-\varepsilon)\}\right )$ such that $g'(\xi_b) = 0$ ($g(z)$ from~\eqref{eq:inverse})}\label{line:support2}
			\State{The support of $\mu$ is $[a,b] := [g(\xi_a), g(\xi_b)]$.}
            \State{\% \emph{Step 3: Compute the Cauchy transform of $\mu$}}
			\State{Find a (large, if possible) circle inside $G_\mu(\mathbb{C} \cup \{\infty\} \backslash [a,b])$ by binary search on the radius, using Theorem~\ref{thm:support} to check whether the circle is inside the wanted region or not. Call $r_B$ its radius and $\Gamma := g(r_B \partial \mathbb{D})$.}
            \State{Let $r_C$ be (slightly) smaller than the radius of the largest circle inscribed in $J^{-}_{[a,b]}(\Gamma)$.}
			\State{Evaluate $\mathcal G_\mu$ on $M = \max\{100, \left \lvert 16 \log_{1/r_C} 10 \right \rvert \}$ equispaced points on the circle of radius $r_C$, using the Cauchy integral theorem + trapezoidal quadrature rule~\eqref{eq:int4}.}\label{line:evalGtilde}
            \State{\% \emph{Step 4: Recover the density of $\mu$}}
			\State{Approximate first $m$ coefficients of power series of $\mathcal G_{\mu}(z)$ via IFFT~\eqref{eq:IFFT}.}
			\State{Return the approximation of the density of $\mu$ by~\eqref{eq:recovery}.}
		\end{algorithmic}
  \end{small}
	\end{algorithm}

 \begin{remark}[Comparison with~\cite{Olver2012}]\label{rmk:compare}
The work~\cite{Olver2012} also proposes an algorithm to compute the free convolution of measures. In the algorithm from~\cite{Olver2012}, in the sqrt-behavior case, they explicitly compute the inverse of the Cauchy transform. However, our approach via the Cauchy integral theorem allows us to compute the R-transform for any measure with compact support, and not only the ones with sqrt-behavior at the boundary. This makes sense to consider, because the measure $\mu_2$ may not have sqrt-behavior at the boundary. In addition, our approach for evaluating $\mathcal G_{\mu}$ on a circle in Step 3 simplifies the approach used in~\cite{Olver2012} to recover the resulting measure from its inverse Cauchy transform.
\end{remark}

 \section{An algorithm for free multiplicative convolution}\label{sec:algorithmtimes}

The case of free \emph{multiplicative} convolution is very similar to the additive case, so we will skip the details. We present an algorithm for two input measures $\mu_1$ and $\mu_2$ that satisfy the following set of assumptions. 

\begin{assumptions}\label{ass:prod}
We assume that $\mu_1$ and $\mu_2$ satisfy the following properties.
\begin{itemize}
    \item They have compact support $[a_1, b_1]$ and $[a_2, b_2]$, with $a_1, b_1, a_2, b_2 > 0$.
    \item Both $\mu_1$ and $\mu_2$ have sqrt-behavior at the boundary.
    \item The T-transforms $T_{\mu_1}$ and $T_{\mu_2}$ are invertible on their domain of definition.
\end{itemize}
\end{assumptions}

Under these assumptions (which are more restrictive than the Assumptions~\ref{ass:sum}), a theorem similar to Theorem~\ref{thm:support} holds. The proof is in Appendix B.

\begin{theorem}\label{thm:supporttimes}
	Under the Assumptions~\ref{ass:prod}, let us define
\begin{equation}\label{eq:inversetimes}
	t(w) := \frac{w}{1+w} T_1^{-1}(w) T_2^{-1}(w),
\end{equation}
which coincides with $T_{\mu}^{-1}(w)$ in a neighborhood of zero. Then the following properties hold.
	\begin{itemize}
		\item $\mu:= \mu_1\boxtimes \mu_2$ has sqrt-behavior at the boundary.
		\item The support of $\mu$ is contained in the interval 
		\begin{equation*}
			[a, b] := [t(\xi_a), t(\xi_b)],
		\end{equation*}
		where $\xi_a$ and $\xi_b$ are the unique zeros of the derivative $t'$ in the intervals $(\max(T_1(a_1), T_2(a_2)), 0)$ and $(0, \min(T_1(b_1), T_2(b_2)))$, respectively.
		\item To check whether a point $w$ is in the image of $T_{\mu_1 \boxtimes \mu_2}$ we can use the following criterion:
		\begin{equation*}
			w \in T_{\mu_1 \boxplus \mu_2}(\mathbb{C} \backslash [a,b]) \Leftrightarrow
			\mathrm{sgn}(\mathrm{Im}(t(w))) = -\mathrm{sgn}(\mathrm{Im}(w)).
		\end{equation*}
	\end{itemize}
\end{theorem}

In summary, the algorithm again has four steps:
\begin{enumerate}
     \item Set up the computation of the S-transform for $\mu_1$ and $\mu_2$.
     \item Compute the support of $\mu$
     \item Compute the T-transform of $\mu$ on a suitable set of points.
     \item Recover the density of $\mu$ from its T-transform.
 \end{enumerate}

\paragraph{Step 1.} Following the same procedure as in Section~\ref{sec:evalG} we can find an approximation of $T(z)$ by rewriting the integral that defines $T$ on the unit circle and then discretizing it using the trapezoidal quadrature rule:
\begin{equation*}
	T(z)  \approx \frac{\pi}{N} \sum_{k=0}^{N}{}^{''} \frac{b-a}{2} \sin\left ( \frac{2\pi k}{N}\right ) \frac{f\left ( \frac{b-a}{2}\cos \left ( \frac{2\pi k}{N}\right ) + \frac{b+a}{2} \right )}{z - \left ( \frac{b-a}{2}\cos \left ( \frac{2\pi k}{N}\right ) + \frac{b+a}{2} \right )}. 
\end{equation*}
In order to be able to evaluate the S-transform (or the inverse of the T-transform), we will need a contour on which to apply the Cauchy integral theorem. In our algorithm, we compute $\mathcal T_{\mu_1}$ and $\mathcal T_{\mu_2}$ in $N$ equispaced points on a circle of radius $r_A<1$ (close to $1$). Analogously, we can compute the derivatives of the T-transform in these same points. Let us denote the approximations by
\begin{equation}\label{eq:cd}
	c_j^{(k)} \approx T_{k}(J_{[a_k, b_k]}(\xi_N^j)), \qquad d_j^{(k)} \approx T_{k}'(J_{[a_k, b_k]}(\xi_N^j))
\end{equation}
for $j=0,1,\ldots,N$ and for $k=1,2$. 

As the S-transform is analytic in a neighborhood of zero, we can write the Cauchy integral formula for any $z$ inside the curve $\Gamma_1 = \mathcal T_{\mu_1}(r_A \partial \mathbb{D})$ (the procedure is analogous for $\Gamma_2$):
\begin{equation*}
	S_{\mu_1}(w) = \frac{1}{2\pi i} \int_{\Gamma_1} \frac{S_{\mu_1}(s)}{w-s} \mathrm{d}s.
\end{equation*}
Using the parametrization of $\Gamma_1$ from the circle of radius $r_A$, and then discretizing the circle in $N$ equispaced points using the trapezoidal quadrature rule, we obtain a formula for approximating $S_{\mu_1}(w)$ (and therefore also $T_{\mu_1}^{-1}(w)$) for any $w$ inside $\Gamma_1$:
\begin{equation*}
	S_{\mu_1}(w) \approx \frac{r_A}{N} \sum_{j=1}^N{}^{''} \frac{\xi_N^j d_j^{(1)} J_{[a_1, b_1]}'(r_A \xi_N^j) (1 + c_j^{(1)})}{c_j^{(1)} J_{[a_1, b_1]}(r_A \xi_N^j) (w - c_j^{(1)})}, \qquad T_{\mu_1}^{-1}(w) = \frac{1+w}{w \cdot S_{\mu_1}(w)}.
\end{equation*}

\begin{algorithm}[htb]
\begin{small}
	\caption{Computation of the free product of two measures}
	\label{alg:multiplication}
	\begin{algorithmic}[1]
		\Require{Measures $\mu_1$ and $\mu_2$, number of quadrature points $N$, number of Fourier coefficients $m$, tolerance $\varepsilon$, radius $r_A$, tolerance \texttt{tol}. Optional: T-transform, S-transform, derivative of T-transform, for one or both measures.}
		\Ensure{Discretization of the density of the measure $\mu := \mu_1 \boxtimes \mu_2$.}
        \State{\% \emph{Step 1: Set up the computation of the S-transform for both measures}}
		\For{$k=1,2$}
		\If{(Inverse) S-transform of $\mu_k$ is unknown}
		\State{Compute values of $d_j^{(k)}, J_{[a_k, b_k]}'(\xi_N^j), c_j^{(k)}$ appearing in Eq.~\eqref{eq:cd}.}
		\EndIf
		\EndFor
        \State{\% \emph{Step 2: Compute the support of $\mu$}}
		\State{Find $\xi_a \in \left (\max\{\mathcal T_{\mu_1}(-r_A+\varepsilon), \mathcal T_{\mu_2}(-r_A+\varepsilon)\}, 0\right )$ such that $t'(\xi_a) = 0$ ($t(z)$ from~\eqref{eq:inversetimes}).}
		\State{Find $\xi_b \in \left (0, \min\{\mathcal T_{\mu_1}(r_A-\varepsilon), \mathcal T_{\mu_2}(r_A-\varepsilon)\}\right )$ such that $t'(\xi_b) = 0$ ($t(z)$ from~\eqref{eq:inversetimes})}
		\State{The support of $\mu$ is $[a,b] := [t(\xi_a), t(\xi_b)]$.}
	       \State{\% \emph{Step 3: Compute the T-transform of $\mu$}}	
  \State{Find a (large, if possible) circle inside $T_\mu(\mathbb{C} \cup \{\infty\} \backslash [a,b])$ by binary search on the radius, using the criterion from Theorem~\ref{thm:supporttimes} to check whether the circle is inside the wanted region or not. Let $r_B$ be the radius of the chosen circle and let $\Gamma:= t(r_B \partial \mathbb{D})$ be the image of the chosen circle via $t$.}
  \State{Let $r_C$ be (slightly) smaller than the radius of the largest circle inscribed in $J_{[a,b]}^{-}(\Gamma)$}
		\State{Evaluate $\mathcal T_\mu$ on $M = \max\{100, \left \lvert 16 \log_{1/r_C} 10 \right \rvert\}$ equispaced points on a circle of radius $r_C$, which is slightly smaller than the largest circle inscribed in $\Gamma$, using the Cauchy integral formula + trapezoidal quadrature rule~\eqref{eq:int4times}.}
        \State{\% \emph{Step 4: Recover the density of $\mu$ from its T-transform.}}
		\State{Approximate first $m$ coefficients of power series of $\mathcal T(z)$ via IFFT~\eqref{eq:IFFT}.}
		\State{Return the approximation of the density of $\mu$ by~\eqref{eq:recoverytimes}.}
	\end{algorithmic}
\end{small}
\end{algorithm}

\paragraph{Step 2.} Being able to evaluate $T_{\mu_1}^{-1}$ and $T_{\mu_2}^{-1}$, we can find the support $[a,b]$ of $\mu := \mu_1 \boxtimes \mu_2$ using Theorem~\ref{thm:supporttimes} (by bisection).

\paragraph{Step 3.} Theorem~\ref{thm:supporttimes} also gives a criterion to find a circle of radius $r_B$ inside the image of $T_{\mu}$. Analogously to Section~\ref{sec:integralG}, we can now evaluate $T_{\mu}$ in any point inside $T_{\mu}^{-1}(r_B \partial \mathbb{D})$ using the Cauchy integral formula. The discretization by the trapezoidal quadrature rule reads
\begin{equation}\label{eq:int4times}
	T_{\mu}(z) \approx \frac{r_B^2}{N} \sum_{j=0}^{N-1} \frac{(J_{[a,b]}^{-})'(t(r_B \xi_N^j)) t'(r_B \xi_N^j) \xi_N^{2j}}{J_{[a,b]}^{-}(t(r_B \xi_N^j)) - z}.
\end{equation}
We use this formula to approximate $\mathcal T_{\mu}(z)$ on a circle of radius $r_C$ contained inside $\mathcal T_{\mu}^{-1}(r_B \partial \mathbb{D})$. 

\paragraph{Step 4.} From the approximate values of $T_{\mu}$, we find a truncated series representation of $\mathcal T_{\mu} (v) \approx \sum_{j =1}^m t_j v^j$ using~\eqref{eq:IFFT} (which also holds when replacing $\mathcal G$ with $\mathcal T$). Finally, we recover the measure $\mu$ from the following relation:
\begin{equation}\label{eq:recoverytimes}
	\begin{bmatrix}
		f\left (J_{[a,b]}(\xi_m^0) \right ) J_{[a,b]}(\xi_m^0)\\
		f\left (J_{[a,b]}(\xi_m^1) \right ) J_{[a,b]}(\xi_m^1)\\
		\vdots\\
		f\left (J_{[a,b]}(\xi_m^{m-1}) J_{[a,b]}(\xi_m^{m-1})\right )
	\end{bmatrix} \approx -\frac{1}{\pi}\mathrm{Im} \left (\begin{bmatrix}
		\xi_m^0 & & & \\
		& \xi_m^1 & & \\
		& & \ddots & \\
		& & & \xi_m^{m-1}	
	\end{bmatrix} \cdot F \cdot \begin{bmatrix} t_1\\t_2\\ \vdots \\ t_m \end{bmatrix}  \right ).
\end{equation}
The proposed method for free multiplicative convolution is summarized in Algorithm~\ref{alg:multiplication}.

\section{Error analysis of free convolutions}\label{sec:erroranalysis}
The discussion of this section focuses on free additive convolution; the proof for free multiplicative convolution is exactly the same. The aim is to explain, for each step of our proposed algorithm, what sources of error are present, and how they can be bounded under suitable assumptions on the measures $\mu_1$ and $\mu_2$ and the other parameters involved in Algorithm~\ref{alg:sum}.

Important for our analysis is the well-known fact that the convergence of the trapezoidal quadrature rule is exponential when the function to be integrated is analytic in a suitable region of the complex plane \cite{davis1959numerical}. We recall a version of this result from~\cite{Trefethen2014}.

\begin{theorem}[Exponential convergence of trapezoidal rule for analytic functions]\label{thm:trapezoidal}
Suppose $u$ is analytic and satisfies $|u(s)| \le C$ in an annulus $\rho^{-1} < |s| < \rho$ of the complex plane, for some $\rho > 1$. Let $N > 1$ and consider the approximation $I_N$ of the integral $I:= \int_{\partial \mathbb{D}} u(s) \mathrm{d}s$ using $N$ equispaced quadrature points on the unit circle and the trapezoidal quadrature rule. Then
		\begin{equation*}
			|I_N - I| \le \frac{4\pi C}{\rho^N - 1}.
		\end{equation*}
	\end{theorem}

This immediately implies the following result on the error in the computation of the Cauchy transforms for measures with sqrt-behavior at the boundary. 
	\begin{corollary}\label{cor:sqrtexp}
		Let $\mu$ be a measure with sqrt-decay behavior at the boundary, and further assume that the function $\psi(J_{[a,b]}(v))$ (from Definition~\ref{def:sqrtbehavior}) is analytic in an annulus $ \mathcal{A}_{\rho_0} = \{\rho_0^{-1} < |v| < \rho_0\}$ around the unit circle. Let $z \in \mathbb{C} \backslash[a,b]$ and let $c$ be the approximation of $G(z)$ obtained by the trapezoidal quadrature rule~\eqref{eq:quadratureG} with $N$ points. Then for any $\rho < \min(\rho_0, |J_{[a,b]}^+(z)|)$ we have
		\begin{equation*}
			|G(z) - c| \le \frac{4\pi}{\rho^N - 1}\cdot \|\psi \circ J_{[a,b]}\|_{\mathcal{A}_{\rho}}\frac{(b-a)^2}{8} \rho \left ( \rho + \frac{1}{\rho}\right )^2 \cdot \frac{1}{\mathrm{dist}(z, J_{[a,b]}(\mathcal{A}_{\rho}))},
		\end{equation*}
  where $\|\psi \circ J_{[a,b]}\|_{\mathcal{A}_{\rho}}$ denotes the maximum modulus of the function $\psi \circ J_{[a,b]}$ on the set $\mathcal{A}_{\rho}$.
	\end{corollary}
\begin{proof}
If $\mu$ has sqrt-behavior at the boundary, we can further manipulate the expression~\eqref{eq:that} and get
	\begin{align*}
		G(z) & = -\frac{i(b-a)^2}{4}\int_0^\pi \sin^2(\theta) \frac{\psi\left ( \frac{b-a}{2}\cos (\theta) + \frac{b+a}{2} \right )}{z - \left ( \frac{b-a}{2}\cos (\theta) + \frac{b+a}{2} \right )} \mathrm{d}\theta\\
		& = \frac{1}{2} \frac{i(b-a)^2}{4}\int_0^{2\pi} \sin^2(\theta) \frac{\psi\left ( \frac{b-a}{2}\cos (\theta) + \frac{b+a}{2} \right )}{z - \left ( \frac{b-a}{2}\cos (\theta) + \frac{b+a}{2} \right )} \mathrm{d}\theta.
	\end{align*}
Here, the second equality follows from symmetry properties of $\sin^2\theta$ and $\cos\theta$. Changing back the variable to $v = \exp(i\theta)$ we get
\begin{equation}\label{eq:int1}
	G(z) = -\frac{1}{2} \int_{\partial \mathbb{D}} \frac{(b-a)^2}{4}\frac{1}{v} \left ( v - \frac{1}{v}\right )^2 \frac{\psi(J_{[a,b]}(v))}{z - J_{[a,b]}(v)} \mathrm{d}v,
\end{equation}
with $\partial\mathbb{D}$ in anti-clockwise direction. Discretizing the integral~\eqref{eq:int1} using the trapezoidal quadrature rule with $2N$ equispaced points is equivalent to discretizing~\eqref{eq:that} in $N+1$ equispaced quadrature points in $[-\pi, 0]$ with the trapezoidal quadrature rule.
	Let us consider the function
	\begin{equation*}
		u(v) := -\frac{1}{2} \frac{(b-a)^2}{4}\frac{1}{v} \left ( v - \frac{1}{v}\right )^2 \frac{\psi(J_{[a,b]}(v))}{z - J_{[a,b]}(v)}
	\end{equation*}
	that appears when computing the Cauchy transform.
	Note that $\frac{1}{v} \left ( v - \frac{1}{v}\right )^2$ is analytic in $\mathbb{C} \backslash \{0\}$ and $\frac{1}{z - J_{{a,b}}(v)}$ is analytic in the annulus $|J_{[a,b]}^{-}(z)| <	|v| < |J_{[a,b]}^{+}(z)|$. The statement of the corollary, therefore, follows from Theorem~\ref{thm:trapezoidal}.
	\end{proof}

 A result for the derivative $G'(z)$ can be obtained in a similar way.
	
In our algorithm, we use the approximation~\eqref{eq:quadratureG} for points on the Joukowski transform of a circle of radius $r_A < 1$, therefore $\rho = \min(\rho_0, r_A^{-1})$ in this case. This means that the closer $r_A$ gets to $1$, the slower the convergence becomes. Note that~\eqref{eq:quadratureG} can also be used for measures that do not have the sqrt-behavior at the boundary. However, we have no guarantees on the fast convergence of the trapezoidal quadrature rule.
	
\begin{theorem}[Error in the computation of the R-transform]\label{thm:errorR}
			Let $\mu$ be a measure with support on $[a,b]$ and let $w$ be a point inside the curve $\Gamma = G(r \partial\mathbb{D})$ for some $r \in (0,1)$ for which $G$ is invertible inside the disk of radius $r$. Let $\rho < \min \{r^{-1}, \mathrm{dist}(\mathcal G^{-1}(w), r\partial \mathbb{D})\}$. Let $\xi_N = \exp(2\pi i / N)$. Assume we have computed approximations $c_j \approx \mathcal G(r\xi_N^j)$ and $d_j \approx G'(J(r\xi_N^j))$ such that $|\mathcal G(r\xi_N^j) - c_j| \le \varepsilon$ and $|G'(J(r\xi_N^j)) - d_j| \le \varepsilon$. Let $m_1 = \min \{\| \mathcal G \|_{\Gamma}, c_0, \ldots, c_{N-1}\}$. Let $m_2$ be the distance of $w$ from $\Gamma$. Let us discretize the integral ~\eqref{eq:int2} with the trapezoidal quadrature rule~\eqref{eq:R1}
			
			\begin{equation*}
				R(w) \approx \frac{r}{N} \sum_{j=1}^N \xi_N^{j} \cdot d_j \cdot J'(\xi_N^j) \cdot \frac{J(\xi_N^j) - 1/c_j}{ c_j- w} =: R_N(w).
			\end{equation*}
			Then
%			\begin{small}
			\begin{multline}\label{eq:resultR}
				|R_N(w) - R(w)| \le \frac{4\pi}{\rho^N -1} \max_{\rho^{-1} \le |v| \le \rho} |u(v)| \\
				+ \varepsilon r \|J'\|_{r \partial \mathbb{D}}\left ( \frac{\|J\|_{r\partial \mathbb{D}}}{m_2} + \frac{1}{m_1 m_2}+ (\varepsilon + \|G'\|_{J(r\partial \mathbb D)})\frac{1}{m_2^2} \left ( \|J\|_{r \partial \mathbb D}  + \frac{|w|}{m_1^2} + \frac{2}{m_1} \right ) \right ).
			\end{multline}
%		\end{small}
\end{theorem}

\begin{proof}
Informally, we have two sources of error: 
\begin{itemize}
    \item the quadrature error -- as the R-transform is analytic, Theorem~\ref{thm:trapezoidal} can be applied;
    \item the fact that the function to be integrated cannot be evaluated exactly -- each point requires quadrature as well.
\end{itemize}
	More precisely, by triangular inequality we have
	\begin{align*}
		|R(w) &- R_N(w)| \le \left \lvert \frac{1}{2\pi i} \int_{\Gamma} \frac{R(s)}{w-s} -  \frac{r}{N} \sum_{j=1}^N \xi_N^{j} \cdot G'(J(r\xi_N^j)) \cdot J'(\xi_N^j) \cdot \frac{J(r\xi_N^j) - 1/\tilde G(r\xi_N^j)}{ \tilde G(r\xi_N^j)- w} \right \vert\\
		& +   \frac{r}{N} \sum_{j=1}^N |\xi_N^{j} J'(r\xi_N^j)| \cdot \left \vert d_j \frac{J(r\xi_N^j) - 1/c_j}{ c_j- w} - G'(J(r\xi_N^j))  \frac{J(r\xi_N^j) - 1/\mathcal G(r\xi_N^j)}{ \mathcal G(r\xi_N^j)- w}\right \rvert.
	\end{align*}
	We can directly use Theorem~\ref{thm:trapezoidal} to bound the first term by $\frac{4\pi}{\rho^N -1}\max\{|u(v)| : \rho^{-1} \le |v| \le \rho\}$, where $u(v)$ is the function to be integrated, defined in~\eqref{eq:int2}. Let us analyze each term in the sum of the second part. To simplify notation, we drop the dependence of the functions on $r\xi_N^j$ and denote $J(r\xi_N^j) = J_j$, $ G'(J(r\xi_N^j)) = G_j'$, $\mathcal G(r\xi_N^j) = \mathcal G_j$. We have %\LY{$\mathcal G_j$ not defined}
	\begin{align*}
		&\left \lvert d_j \frac{J_j - 1/c_j}{c_j-w} - G'_j \frac{J_j - 1/\mathcal G_j}{\mathcal G_j - w} \right \vert  \le | G'_j - d_j| \cdot \left \lvert \frac{J_j - 1/\mathcal G_j}{\mathcal G_j - w} \right \rvert + |d_j| \cdot \left \vert \frac{J_j - 1/c_j}{c_j - w} - \frac{J_j - 1/\mathcal G_j}{\mathcal G_j - w} \right \rvert\\
		& = | G'_j - d_j| \cdot \left \lvert \frac{J_j - 1/\mathcal G_j}{\mathcal G_j - w} \right \rvert + |d_j| \cdot \left \vert \frac{J_j(\mathcal G_j - c_j) + w(1/c_j - 1/\mathcal G_j) + (c_j/\mathcal G_j - \mathcal G_j c_j)}{(c_j - w)(\mathcal G_j - w)}  \right \rvert.
	\end{align*}
We can bound the terms with $| G'_j - d_j| \cdot \left \lvert \frac{J_j - 1/\mathcal G_j}{\mathcal G_j - w} \right \rvert \le \frac{\varepsilon}{m_2}\left ( \|J\|_{r \partial \mathbb{D}} + \frac{1}{m_1} \right )$, $|1/c_j - 1/\mathcal G_j|  = \left \lvert \frac{\mathcal G_j - c_j}{c_j \mathcal G_j} \right \vert \le \frac{\varepsilon}{m_1^2}$, $|c_j/\mathcal G_j - \mathcal G_j /c_j| \le \frac{2\varepsilon}{m_1}$, $|(c_j - w)(\mathcal G_j - w)| \ge m_2^2 $, therefore obtaining~\eqref{eq:resultR}.
% \begin{itemize}
% \item $| G'_j - d_j| \cdot \left \lvert \frac{J_j - 1/\mathcal G_j}{\mathcal G_j - z} \right \rvert \le \frac{\varepsilon}{m_2}\left ( \|J\|_{r \partial \mathbb{D}} + \frac{1}{m_1} \right )$
% \item $|1/c_j - 1/\mathcal G_j|  = \left \lvert \frac{\mathcal G_j - c_j}{c_j \mathcal G_j} \right \vert \le \frac{\varepsilon}{m_1^2}$
% \item $|c_j/\mathcal G_j - \mathcal G_j /c_j| \le \frac{2\varepsilon}{m_1}$
% \item $|(c_j - z)(\mathcal G_j - z)| \ge m_2^2 $
% \end{itemize}
% Putting everything together we obtain~\eqref{eq:resultR}.
\end{proof}

Theorem~\eqref{thm:errorR} basically states that the quadrature error in the computation of the R-transform decreases exponentially with the number of quadrature points, and the error due to the inexact evaluation of $G$ and $G'$ is not amplified significantly. Note that this result does not need the assumption that $\mu$ has sqrt-behavior at the boundary.
	
\begin{remark}[Error in the computation of the support of $\mu_1 \boxplus \mu_2$]
Let us denote by $g_N'(w)$ the approximation of the function $g'(w)$ obtained using $N$ quadrature points for all the involved quadrature rules. When $N \to \infty$, the functions $g_N'(w)$ are holomorphic and they converge to $g'(w)$ uniformly on compact subsets of the intervals $\left (\max\{\mathcal G_{\mu_1}(-r_A+\varepsilon), \mathcal G_{\mu_2}(-r_A+\varepsilon)\}, 0\right )$ and $\left (0, \min\{\mathcal G_{\mu_1}(r_A-\varepsilon), \mathcal G_{\mu_2}(r_A-\varepsilon)\}\right )$. By Hurwitz's theorem, we can conclude that for sufficiently large values of $N$, the functions $g_N'(w)$ have exactly one simple zero in each of the intervals, and these zeros converge to the zeros $\xi_a, \xi_b$ of $g'(w)$. Therefore, when the number of quadrature points is sufficiently large, the computed values of the support $[a,b]$ of $\mu$ are accurate.
\end{remark}

What we said until now means that we are able to evaluate the R-transforms of $\mu_1$ and $\mu_2$ accurately, together with the support of $\mu = \mu_1 \boxplus \mu_2$. The next step in Algorithm~\ref{alg:sum} is the evaluation of $\mathcal G_{\mu}(v)$ on a suitable circle of radius $r_C \in (0,1)$ (see line~\ref{line:evalGtilde}).
    
\begin{remark}[Error in the computation of the Cauchy transform of $\mu_1 \boxplus \mu_2$]
Let $r = r_C$ and let us discuss the sources of error in the evaluation of $\mathcal G(v)$. Let us denote by $\widetilde g_j, \widetilde g_j'$ the approximations of $g(r \xi_N^j)$ and $g'(r \xi_N^j)$, respectively, obtained by numerical quadrature. For any point $v$ inside the curve $J_{[a,b]}^-(\Gamma)$ defined in Section~\ref{sec:integralG}, letting $\widetilde a, \widetilde b$ be the approximated extrema of the support of $\mu$, the approximation of $\mathcal G(v)$ obtained from~\eqref{eq:int4} would be %\LY{$\widetilde a, \widetilde b$ not defined.} 
%\LY{superscript comp is quite unusual. can we replace it with something else?}
\begin{equation*}
  \widetilde {\mathcal G}_v := \frac{r^2}{N} \sum_{j=0}^{N-1} \frac{\widetilde g_j' \xi_N^{2j}}{\widetilde g_j - J_{[\widetilde a, \widetilde b]}(v)}.
  \end{equation*}
  We can split the error $|\mathcal G(v) -  \widetilde {\mathcal G}_v|$ into three sources by triangular inequality:
  \begin{align*}
      |\mathcal G(v) - \widetilde {\mathcal G}_v| & \le \left \lvert G(J_{[a,b]}(v)) - G(J_{[\widetilde a,\widetilde b]}(v)) \right \vert + \left \lvert  G(J_{[\widetilde a,\widetilde b]}(v)) - \frac{r^2}{N}\sum_{j=0}^{N-1} \frac{g'(r\xi_N^j) \xi_N^{2j}}{g(r \xi_N^j)-J_{[\widetilde a,\widetilde b]}(v)}\right \rvert \\ &+ \frac{r^2}{N} \sum_{j=0}^{N-1}  \left \lvert \frac{g'(r\xi_N^j) }{g(r \xi_N^j)-J_{[\widetilde a,\widetilde b]}(v)} - \frac{\widetilde g_j'}{\widetilde g_j-J_{[\widetilde a,\widetilde b]}(v)} \right \rvert.
  \end{align*}
  We refrain from outlining a full error analysis of this step, but we will briefly comment on the three terms in the above error bound. The first term can be expected to be small if $[\widetilde a, \widetilde b]$ is accurate. The second term, the error in the trapezoidal quadrature rule, can be bounded by Theorem~\ref{thm:trapezoidal} because the function to be integrated (the Cauchy transform $G$) is analytic in an annulus around $\partial \mathbb{D}$; therefore, this part decreases exponentially fast in the number of the quadrature points $N$. Finally, the third term can be bounded similarly to what is done in Theorem~\ref{thm:errorR} if we know that $\widetilde g_j, \widetilde g_j', \widetilde a, \widetilde b$ have been computed accurately.
	\end{remark}

The last part of the algorithm is about recovering the density of the measure $\mu = \mu_1 \boxplus \mu_2$ from the (approximate) values of the Cauchy transform $G = G_{\mu}$ on equispaced points on the circle of radius $r = r_C$.
 
	\begin{theorem}[Error in the computation of the coefficients of $\mathcal G$]\label{thm:recoverg}
		Let $\mu$ be a measure with sqrt-behavior, at the boundary, with support $[a,b]$ and density $\mathrm{d}\mu(x) = \sqrt{b-x}\sqrt{x-a}\psi(x) \mathrm{d}x$, where $\psi(x)$ is Lipschitz-continuous. Assume that we have a vector $\begin{bmatrix} \widetilde {\mathcal{G}}_0 & \widetilde {\mathcal{G}}_1 & \cdots & \widetilde {\mathcal{G}}_{M-1} \end{bmatrix}^T$ that approximates, up to entrywise error $\varepsilon$, the vector  $\begin{bmatrix} \mathcal G(r \xi_M^0) & \mathcal G(r \xi_M^1) & \cdots & \mathcal G(r \xi_M^{M-1}) \end{bmatrix}^T$ for some $M \in \mathbb{N}$ and some $r<1$, that is, $|\widetilde{\mathcal{G}}_j - \mathcal G(r \xi_M^j)| \le \varepsilon$ for $j=0,1,\ldots,M-1$. Let
%		\begin{scriptsize}
		\begin{equation*}
			\begin{bmatrix}
				\widetilde g_1\\
				\widetilde g_2\\
				\vdots\\
				\widetilde g_M
			\end{bmatrix} := \begin{bmatrix} r^{-1} & & & \\ & r^{-2} & & \\ & & \ddots & \\ & & & r^{-M} \end{bmatrix} F^{-1}\begin{bmatrix} \xi_M^{-0} & & & \\ & \xi_M^{-1} & & \\ & & \ddots & \\ & & & \xi_M^{-M+1} \end{bmatrix} \begin{bmatrix} \widetilde{\mathcal{G}}_0 \\ \widetilde{\mathcal{G}}_1 \\ \vdots \\ \widetilde{\mathcal{G}}_{M-1} \end{bmatrix}.
		\end{equation*}
	Then
	\begin{equation*}
		\left \lvert \begin{bmatrix}
			\widetilde g_1\\
			\widetilde g_2\\
			\vdots\\
			\widetilde g_M
		\end{bmatrix} - \begin{bmatrix}
		g_1\\
		g_2\\
		\vdots\\
		g_M
	\end{bmatrix} \right \rvert \le \begin{bmatrix}
	r^{-1}\\
	r^{-2}\\
	\vdots\\
	r^{-M}
\end{bmatrix} \frac{\varepsilon}{\sqrt M} + \frac{r^M}{1-r^M} \|\psi\|_{[a,b]} \begin{bmatrix} 1\\ 1\\ \vdots \\ 1 \end{bmatrix},
	\end{equation*}
	where the symbol $|\cdot|$ means that the inequality holds entry-wise and $g_1,\ldots, g_M$ are the exact coefficients of the series expansion~\eqref{eq:laurent}.
	\end{theorem}

\begin{proof}
	First of all, we can write 
	\begin{equation*}
		\begin{bmatrix} \mathcal G(r \xi_M^0) \\ \mathcal G(r \xi_M^1) \\ \cdots \\ \mathcal G(r \xi_M^{M-1}) \end{bmatrix} = \begin{bmatrix} \xi_M^{0} & & & \\ & \xi_M^{1} & & \\ & & \ddots & \\ & & & \xi_M^{M-1} \end{bmatrix} \begin{bmatrix} F & F & F & \cdots \end{bmatrix} \begin{bmatrix} g_1 r\\ g_2 r^2\\ g_3 r^3 \\ \vdots \end{bmatrix},
	\end{equation*}
	where we have an infinite matrix and an infinite vector (it is alright to write this because the series converges absolutely). Using triangular inequality (entry-wise), we have 
	
	\resizebox{.8\linewidth}{!}{
		\begin{minipage}{\linewidth}
	\begin{multline*}
		\left \lvert \begin{bmatrix}
			\widetilde g_1\\
			\widetilde g_2\\
			\vdots\\
			\widetilde g_M
		\end{bmatrix} - \begin{bmatrix}
			g_1\\
			g_2\\
			\vdots\\
			g_M
		\end{bmatrix} \right \rvert \le \left \lvert \begin{bmatrix} r^{-1} & & & \\ & r^{-2} & & \\ & & \ddots & \\ & & & r^{-M} \end{bmatrix} F^{-1}\begin{bmatrix} \xi_M^{-0} & & & \\ & \xi_M^{-1} & & \\ & & \ddots & \\ & & & \xi_M^{-M+1} \end{bmatrix} \left (\begin{bmatrix} \widetilde{\mathcal{G}}_0 \\ \widetilde{\mathcal{G}}_1 \\ \vdots \\ \widetilde{\mathcal{G}}_{M-1} \end{bmatrix} - \begin{bmatrix} \mathcal G(r \xi_M^0) \\ \mathcal G(r \xi_M^1) \\ \vdots \\ \mathcal G(r \xi_M^{M-1}) \end{bmatrix} \right ) \right \rvert \\
	+ \left \lvert \begin{bmatrix} r^{-1} & & & \\ & r^{-2} & & \\ & & \ddots & \\ & & & r^{-M} \end{bmatrix} F^{-1} \begin{bmatrix} \xi_N^{-0} & & & \\ & \xi_M^{-1} & & \\ & & \ddots & \\ & & & \xi_M^{-M+1} \end{bmatrix} \begin{bmatrix} \xi_M^{0} & & & \\ & \xi_M^{1} & & \\ & & \ddots & \\ & & & \xi_M^{M-1} \end{bmatrix} \begin{bmatrix} F & F & F & \cdots \end{bmatrix} \begin{bmatrix} g_1 r\\ g_2 r^2\\ g_3 r^3 \\ \vdots \\ \vdots \end{bmatrix} - \begin{bmatrix}
		g_1\\
		g_2\\
		\vdots\\
		g_M
	\end{bmatrix} \right \rvert.
	\end{multline*}
\end{minipage}}

The first term is bounded by
\begin{equation*}
	\left \lvert \begin{bmatrix} r^{-1} & & & \\ & r^{-2} & & \\ & & \ddots & \\ & & & r^{-M} \end{bmatrix} \cdot \left \lVert F^{-1} \begin{bmatrix} \xi_M^{-0} & & & \\ & \xi_M^{-1} & & \\ & & \ddots & \\ & & & \xi_M^{-M+1} \end{bmatrix} \right \rVert_2 \cdot  \begin{bmatrix} \varepsilon \\ \varepsilon \\ \vdots \\ \varepsilon \end{bmatrix} \right \rvert  \le \begin{bmatrix}
		r^{-1}\\
		r^{-2}\\
		\vdots\\
		r^{-M}
	\end{bmatrix} \frac{\varepsilon}{\sqrt N}
\end{equation*}
because $F^{-1}$ has norm $1/\sqrt{M}$ and all the $\xi_M^{-j}$ have modulus 1. The second term simplifies to
\begin{equation*}
	\left \lvert \begin{bmatrix} 
		g_{M+1} r^M + g_{2M+1}r^{2M} + g_{3M+1}r^{3M} + \ldots\\
		g_{M+2} r^M + g_{2M+2}r^{2M} + g_{3M+2}r^{3M} + \ldots\\
		\vdots\\
		g_{2M} r^M + g_{3M}r^{2M} + g_{4M}r^{3M} + \ldots\\
	\end{bmatrix} \right \rvert \le \frac{r^M}{1-r^M} \left ( \max_{k \ge M+1} |g_k| \right ) \begin{bmatrix} 1\\ 1\\ \vdots \\ 1 \end{bmatrix}.
\end{equation*}
Due to the connections with expansions of $\mathcal G$ in the Chebyshev series (see Remark~\ref{rmk:convergenceGtilde}), all coefficients $g_k$ have modulus bounded by $\|\psi\|_{[a,b]}$, which concludes the proof of the theorem.
\end{proof}

Theorem~\ref{thm:recoverg} tells us that whenever $\varepsilon / r^j$ is large, we cannot expect the approximation of $g_j$ to be reliable. This is why we are only keeping the first few approximations of the series coefficients. We also need to ensure that $r^M = r_C^M$ is small so that the second error term is small (but this is less problematic): for this reason, in our code, we set $M = \max\left \{ 100, \left \lvert 16 \log_{1/r_C} 10 \right \rvert \right \}$.

\begin{remark}[Error in the recovery of the density of $\mu_1 \boxplus \mu_2$]
	Let us denote by \\$\begin{bmatrix}
				\widetilde f_0 &
				%\widetilde f_1 &
				\cdots &
				\widetilde f_{m-1}
			\end{bmatrix}^T$ the approximation obtained from~\eqref{eq:recovery} by plugging in the vector $\begin{bmatrix}
				\widetilde g_1 &
				%\widetilde g_2 &
				\cdots &
				\widetilde g_m
			\end{bmatrix}^T$. Using the triangular inequality and the fact that $F$ is unitary up to a $\sqrt{m}$ scaling factor, we get that
 \begin{multline*}
      \left \lvert \begin{bmatrix}
			f\left (J_{[a,b]}(\xi_m^0) \right )\\
		f\left (J_{[a,b]}(\xi_m^1) \right )\\
			\vdots\\
			f\left (J_{[a,b]}(\xi_m^{m-1}) \right )
		\end{bmatrix} -  \begin{bmatrix} \widetilde f_0\\ \widetilde f_1\\ \vdots\\ \widetilde f_{m-1} \end{bmatrix} \right \rvert  \le \frac{\sqrt{m}}{\pi} \left \lvert \begin{bmatrix} g_1\\g_2\\\vdots \\g_m \end{bmatrix} - \begin{bmatrix}
				\widetilde g_1\\
				\widetilde g_2\\
				\vdots\\
				\widetilde g_m
			\end{bmatrix}\right \rvert \\ + \left \vert \begin{bmatrix}
			f\left (J_{[a,b]}(\xi_m^0) \right )\\
		f\left (J_{[a,b]}(\xi_m^1) \right )\\
			\vdots\\
			f\left (J_{[a,b]}(\xi_m^{m-1}) \right )
		\end{bmatrix} + \frac{1}{\pi}\mathrm{Im} \left (\begin{bmatrix}
			\xi_m^0 & & & \\
			& \xi_m^1 & & \\
			& & \ddots & \\
			& & & \xi_m^{m-1}	
		\end{bmatrix} \cdot F \cdot \begin{bmatrix} g_1\\g_2\\ \vdots \\ g_m \end{bmatrix}  \right ) \right \vert.
 \end{multline*}
 The first error term can be bounded with Theorem~\ref{thm:recoverg}. The second error term depends on how fast the series~\eqref{eq:laurent} converges to $\mathcal G_{\mu}(v)$, which, in turn, depends on the regularity of the density function associated with $\mu$ (see Remark~\ref{rmk:convergenceGtilde}).
\end{remark}

%====================================================================	
\section{Numerical experiments}\label{sec:experiments}

\subsection{Exponential convergence of trapezoidal quadrature rule for the pointwise evaluation of the Cauchy transform}

\begin{figure}[htb]
		\centering
		\includegraphics[width=\textwidth]{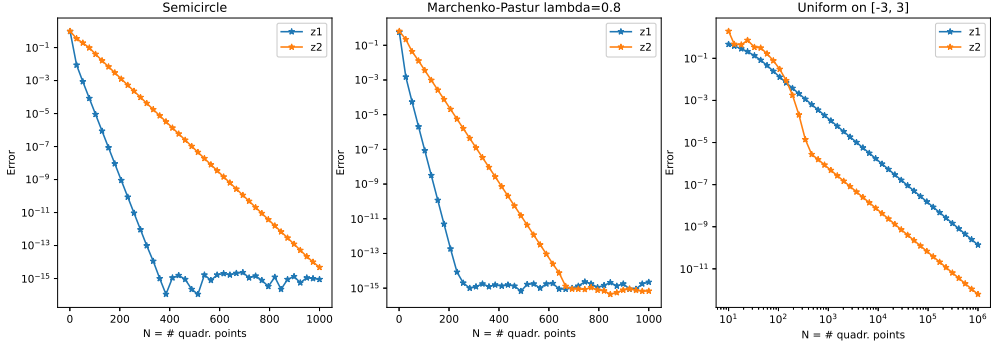}
		\caption{Evaluation of the Cauchy transform of the points $z_1 = 1.001 \cdot b$ and $z_2 = (0.8 + 0.01i)\cdot b$ for the semicircle distribution (left), the Marchenko-Pastur with parameter $0.8$ (middle), and the uniform distribution on $[-3, 3]$ (right). The x-axis is the number of quadrature points, and the y-axis is the error of the trapezoidal quadrature rule.}
		\label{fig:quadrature}
\end{figure}

First of all, we consider the evaluation of the Cauchy transform in some points outside the support for the semicircle, uniform (with $m=3$), and Marchenko-Pastur (with $\lambda = 0.7$) distributions. The results are reported in Figure~\ref{fig:quadrature}. For each measure, we consider the point $z_1 = 1.001 \cdot b$ (blue) and  $z_2 = (0.8 + 0.01i)\cdot b$ (orange). The x-axis is the number of quadrature points, and the y-axis is the error of the trapezoidal quadrature rule. As expected, the error decreases exponentially for the semicircle and Marchenko-Pastur. The same cannot be said for the uniform distribution (note that the graph is log-log in this case).

\subsection{Recovery of measures from the Cauchy transform}
\begin{figure}[htb!]
		\centering
		\includegraphics[width=0.8\textwidth]{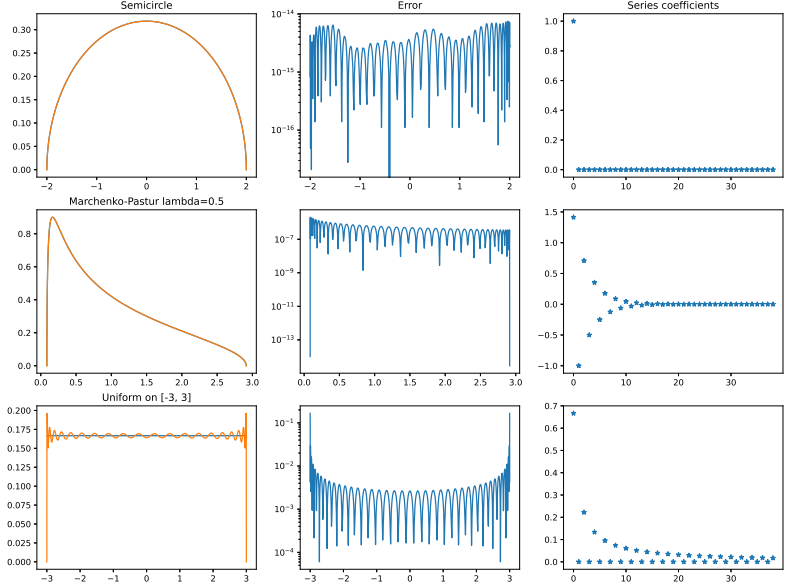}
		\caption{Recovering the measure from $40$ Fourier coefficients of the Cauchy transform. From top to bottom, we consider the semicircle distribution, the Marchenko-Pastur law with parameter $\lambda=0.5$, and the uniform distribution on the interval $[-3,3]$. In the left block column, the blue curve is the density function, and the orange curve is the reconstruction of the measure. In the middle block column, we plot the pointwise error $|f(x) - \mathcal f(x)|$. The right block column illustrates the magnitude of the approximate coefficients $g_1, g_2, \ldots, g_{10}$.}
		\label{fig:m40}
\end{figure}
	
We test numerically the recovery of the spectral measure $\mu$ from the knowledge of the (exact) $\mathcal G(z)$ on a circle of radius $R = 0.9$ in $N = 3000$ points coming from equispaced points on a circle, using $m = 40$ coefficients in the power series~\eqref{eq:laurent} for the Cauchy transform. The results are shown in Figure~\ref{fig:m40}. As expected, we can recover the measures with sqrt-behavior at the boundary very well.

% 	\begin{figure}[htb]
% 		\centering
% 		\includegraphics[width=\textwidth]{Figures/recover_mu30.eps}
% 		\caption{Recovering the measure from $30$ Fourier coefficients of the Cauchy transform.}
% 		\label{fig:m30}
% 	\end{figure}

% \begin{figure}[htb]
% 	\centering
% 	\includegraphics[width=\textwidth]{Figures/recover_mu20times.eps}
% 	\caption{Recovering the measure from $20$ Fourier coefficients of the T-transform. Here we consider a schifted semicircle distribution (with support on $[1, 5]$), the Marchenko-Pastur distribution with $\lambda=0.7$, and the uniform distribution on the interval $[1,2]$. As before, in the first block column we plot the density (blue) and the approximate reconstruction of the density (orange). In the second block column we have the pointwise error and in the third column we have the magnitude of the computed coefficients of the power series of $\mathcal T(z)$. Note that, in the case of the Marchenko-Pastur law, we have $\mathcal T(z) = \frac{z}{\lambda}$.}
% 	\label{fig:m20times}
% \end{figure}
	
% 	\FloatBarrier
	
\subsection{Free additive and multiplicative convolution of distributions}

In this section, we test Algorithm~\ref{alg:sum} on various input measures. In some cases, we know the analytic form of the support of the free sum or the analytic form of the free sum itself. This allows us to test the accuracy of our algorithm on some examples. In the other cases, we will compare the approximate density returned by our algorithm with the empirical eigenvalue distribution of the sum of two random matrices that are chosen according to the input distributions.

\paragraph{How to read the plots.} Let us explain what is plotted in each figure. The first (top) block row is about the measure $\mu_1$, the second block row is about $\mu_2$, and the third block row is about the free convolution $\mu_1 \boxplus \mu_2$. We pass from the first (left) column to the second (middle) column with the Joukowski transform, and from the second column to the third (right) column with the Cauchy/T-transform. The \textcolor{red}{red} circle in the left column is always the unit circle. The \textcolor{red}{red} line in the second column indicates the support of the corresponding measure. The \textcolor{blue}{blue} curve in the (1,3) block (the top right block) is $\Gamma_1$ and the \textcolor{blue}{blue} curve in the (2,3) block (the top middle block) is $\Gamma_2$. These same curves are reported in the (4,3) block, together with an \textcolor{orange}{orange} circle that is in the intersection of the two curves and inside the image of $\mathbb{C} \cup \{\infty\} \backslash[a,b]$. The same \textcolor{orange}{orange} circle is drawn in the (1,3), (2,3), and (3,3) blocks, and the inverses via $G_1, G_2, G$ are in the (1,2), (2,2), (3,2) blocks, respectively. Finally, the circle $\mathcal{A}$ is represented in \textcolor{green}{green} in the (3,1) block, and its image via $J$ and then $G$/$T$ is also in \textcolor{green}{green} in the (3,2) and (3,3) blocks. In the (4,2) block, we plot the densities of the measures $\mu_1$, $\mu_2$, and $\mu$. In the (4,1) block, we (usually) plot the eigenvalue distribution of one instance of a random matrix of size $5000 \times 5000$ obtained as a sum/product of matrices with limit eigenvalue distributions $\mu_1$ and $\mu_2$.

 \paragraph{Implementation details.}
	Unless indicated otherwise, we set the following parameters for Algorithm~\ref{alg:sum}: 
	\begin{itemize}
 \item For computing the Cauchy transform in a point, we use $400$ quadrature points for measures with sqrt-behavior at the boundary (semicircle, Marchenko-Pastur) and $4000$ points for other measures.
 \item We use $400$ quadrature points to discretize the integrals corresponding to the Cauchy integral formulae for computing the R-transform of $\mu_1$ and $\mu_2$ as well as the Cauchy transform of $\mu$.
 \item The number of points on the circle of radius $r_C$ is chosen as $\max\left \{100, \left \lfloor16 \log_{1/r_C}10 \right \rfloor\right \}$ (this is to ensure that the second term in the error of Theorem~\ref{thm:recoverg} is negligible). Then we truncate the series to the first $m = 20$ terms.
 \item We choose a parameter $\varepsilon = 0.05$ that quantifies how ``far from the border'' are the curves and intervals we take. More specifically:
 \begin{itemize}
 \item we set $r_A = 1 - \varepsilon$; 
 \item we use this value of $\varepsilon$ in lines~\ref{line:support1}--\ref{line:support2} of Algorithm~\ref{alg:sum}; 

 \item when looking for the radius of the largest circle inside 
 \begin{equation*}
 G_{\mu_1}(\mathbb{C} \cup \{\infty\} \backslash [a_1, b_1]) \cap G_{\mu_2}(\mathbb{C} \cup \{\infty\} \backslash [a_2, b_2]) \cap G_{\mu}(\mathbb{C} \cup \{\infty\} \backslash [a, b])
 \end{equation*}
 we start looking at the circle of radius $(1-\varepsilon)r$, where $r$ is the radius of the largest circle inscribed in $G_{\mu_1}(\mathbb{C} \cup \{\infty\} \backslash [a_1, b_1]) \cap G_{\mu_2}(\mathbb{C} \cup \{\infty\} \backslash [a_2, b_2])$;
 \item when choosing the radius $r_C$ we multiply by $(1-\varepsilon)$ the radius of the largest circle inscribed in $J_{[a,b]}^{-}(\Gamma)$.
 \end{itemize}
 \item The tolerance for stopping the binary search for the computation of the support of $\mu$ is set to $10^{-12}$. The tolerance for stopping the binary search for the computation of $r_B$ is $10^{-3}$.
	\end{itemize}
For Algorithm~\ref{alg:sum}, we use the same values of the parameters in the corresponding parts of the code.

\paragraph{Code availablity.}
A Python implementation of Algorithms~\ref{alg:sum} and~\ref{alg:multiplication} is available at \url{https://github.com/Alice94/FreeConvolutionCode}, together with scripts to reproduce all the numerical examples in this paper.

\begin{example}\label{ex:semicircle}
In Figure~\ref{fig:semicircle+semicircle}, we compute the free additive convolution of two semicircle distributions. The result is known analytically, and it is a rescaled semicircle distribution, with support in $[-2\sqrt{2}, 2\sqrt{2}]$ and density function $f(x) = \frac{1}{4\pi}\sqrt{8-x^2}$. Our algorithm is able to compute the free additive convolution with an accuracy of $10^{-14}$. We only used $m=10$ coefficients of the series~\eqref{eq:laurent} in this example.
		\end{example}
	
	\begin{figure}[htb]
		\centering
		\includegraphics[width=0.7\textwidth]{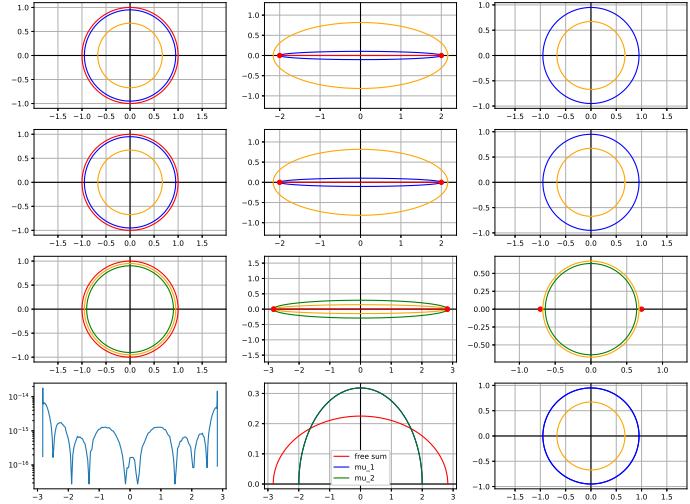}
		\caption{Free sum of semicircle and semicircle. Only 400 quadrature points. The (4, 2) subplot is the error of the computed free sum. See Example~\ref{ex:semicircle} for a precise description.}
		\label{fig:semicircle+semicircle}
	\end{figure}
	
%	\FloatBarrier
	
	\begin{example}\label{ex:freesumstepsemicircle}
		In Figure~\ref{fig:semiunif4}, we compute the free additive convolution of the semicircle distribution with the uniform distribution on the interval $[-4, 4]$. In this case, we do not know the analytic form of the solution, but we can compute the support of the free sum, which is
		\begin{equation*}
			\left [ -\frac{1}{4}\log \left ( 4 + \sqrt{17}\right ) - \sqrt{17}, \frac{1}{4}\log \left ( 4 + \sqrt{17}\right ) + \sqrt{17} \right ].
		\end{equation*}
	The supports that we computed numerically agree with the theoretical one up to an absolute error of $3.6 \cdot 10^{-6}$. The uniform distribution does not have an sqrt-behavior at the boundary; hence the trapezoidal quadrature rule for computing the Cauchy transform converges more slowly. For this reason, we use $4000$ quadrature points for this example. Note that we could have used an explicit expression of the Cauchy transform because it is known analytically, but we wanted to see how our algorithm performed without this knowledge.
	\end{example}

	% \begin{figure}[htb]
	% 	\centering
	% 	\includegraphics[width=0.7\textwidth]{Figures/Semicircle+Unif1.eps}
	% 	\caption{Free sum of semicircle and uniform on $[-1,1]$. See Example~\ref{ex:freesumstepsemicircle}.}
	% 	\label{fig:semiunif1}
	% \end{figure}
	\begin{figure}[htb!]
		\centering
		\includegraphics[width=0.7\textwidth]{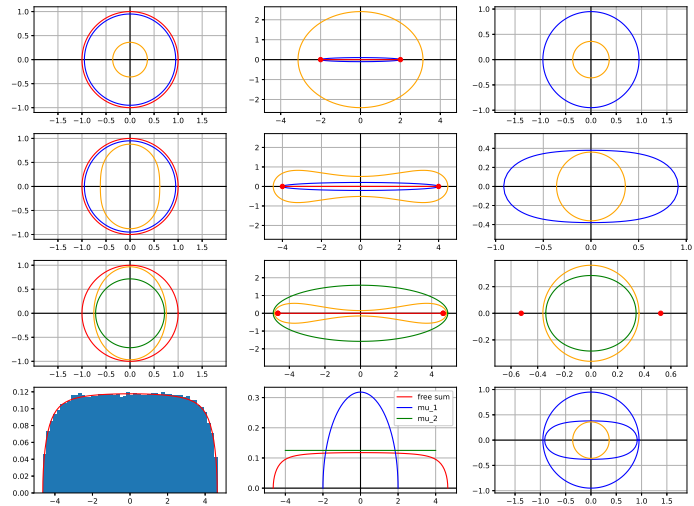}
		\caption{Free sum of semicircle and uniform on $[-4,4]$. See Example~\ref{ex:freesumstepsemicircle}.}
		\label{fig:semiunif4}
	\end{figure}

%\FloatBarrier

%\clearpage
% \begin{example}\label{ex:semicircleMP}
% 	In Figure~\ref{fig:mp1} we consider the free convolution of the semicircle distribution with the Marchenko-Pastur distribution with $\lambda=0.5$.
% \end{example}

% \begin{figure}[htb]
% 	\centering
% 	\includegraphics[width=0.7\textwidth]{Figures/Semicircle+MP0.5.eps}
% 	\caption{Free sum of semicircle and Marchenko-Pastur ($\lambda=0.5$). See Example~\ref{ex:semicircleMP}.}
% 	\label{fig:mp1}
% \end{figure}

%\FloatBarrier
%
%\clearpage
\begin{example}\label{ex:unifMP}
	In Figure~\ref{fig:mp2}, we consider the free additive convolution of the uniform distribution on the interval $[-2, 2]$ and the Marchenko-Pastur distribution with $\lambda=0.7$.
\end{example}
	\begin{figure}[htb!]
		\centering
		\includegraphics[width=0.7\textwidth]{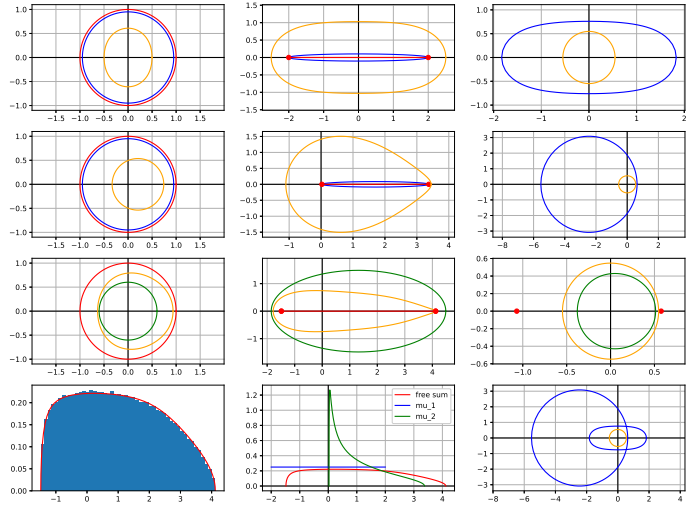}
		\caption{Free sum of uniform on $[-2,2]$ and Marchenko-Pastur ($\lambda=0.7$). See Example~\ref{ex:unifMP}.}
		\label{fig:mp2}
	\end{figure}

%\FloatBarrier

%\clearpage

%\FloatBarrier

%\clearpage

\begin{example}\label{ex:semicirclexMP}
	In Figure~\ref{fig:semicircleMP}, we consider the free multiplicative convolution of a shifted semicircle distribution (with support in $[1, 5]$ instead of $[-2, 2]$) and a Marchenko-Pastur distribution with $\lambda = 0.2$.
\end{example}
\begin{figure}[htb!]
	\centering
	\includegraphics[width=0.7\textwidth]{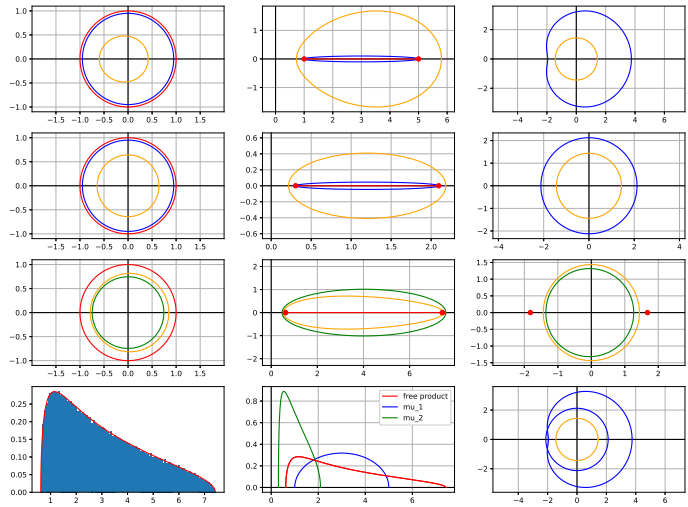}
	\caption{Free multiplication of shifted semicircle and Marchenko-Pastur distribution. See Example~\ref{ex:semicirclexMP}.}
	\label{fig:semicircleMP}
\end{figure}

\begin{example}\label{ex:sxs}
	In Figure~\ref{fig:sxs}, we consider the free multiplicative convolution of two identical shifted semicircle distributions.
\end{example}
\begin{figure}[htb!]
	\centering
	\includegraphics[width=0.7\textwidth]{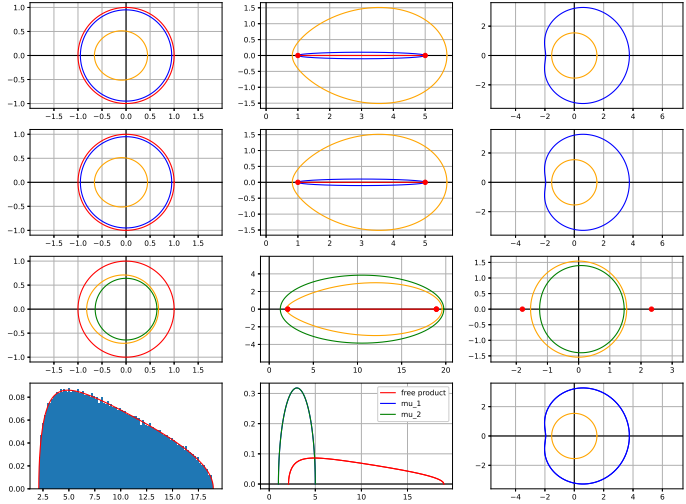}
	\caption{Free multiplication of two shifted semicircle distributions. See Example~\ref{ex:sxs}.}
	\label{fig:sxs}
\end{figure}

% \begin{example}\label{ex:unifxunif}
% 	In Figure~\ref{fig:unifxunif} we consider the free multiplicative convolution of a uniform distribution on the interval $[1, 2]$ and a uniform distribution on the interval $[1, 3]$. We use $4000$ quadrature points for this example.
% \end{example}
% \begin{figure}[htb]
% 	\centering
% 	\includegraphics[width=0.7\textwidth]{Figures/Unif1xUnif2.eps}
% 	\caption{Free multiplication of a uniform distribution on the interval $[1, 2]$ and a uniform distribution on the interval $[1, 3]$. See Example~\ref{ex:unifxunif}.}
% 	\label{fig:unifxunif}
% \end{figure}

 % \FloatBarrier
\subsection{Strengths and limitations of our algorithms}

The numerical examples above show that Algorithms~\ref{alg:sum} and~\ref{alg:multiplication} give a good approximation of the free convolution of measures in a variety of situations. In the following four examples, we show that our algorithm can give good insights even in cases in which the input measure does not satisfy the assumptions of Theorem~\ref{thm:support}. 

\begin{example}\label{ex:unif}
	In Figure~\ref{fig:unif}, we consider the free additive convolution of a uniform distribution on the interval $[-1, 1]$ and a uniform distribution on the interval $[-2, 2]$. This combination does not fall into the theory for our algorithm, but their free additive convolution has an sqrt-behavior at the boundary (see, e.g.,~\cite{Bao2020}). 
\end{example}

\begin{figure}[htb!]
		\centering
		\includegraphics[width=0.7\textwidth]{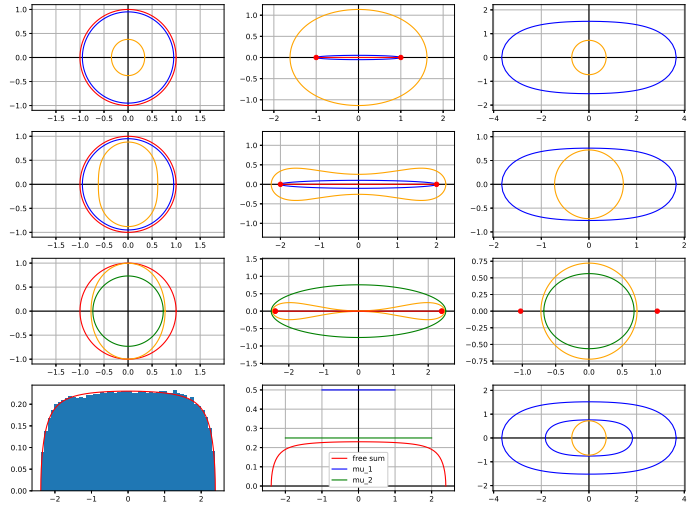}
		\caption{Free sum of uniform distributions on $[-1,1]$ and $[-2,2]$. See Example~\ref{ex:unif}.}
		\label{fig:unif}
\end{figure}

\begin{example}\label{ex:MPxunif}
	In Figure~\ref{fig:MPxunif} we consider the free multiplicative convolution of the Marchenko-Pastur distribution with $\lambda = 0.1$ and the uniform distribution on the interval $[1, 3]$.
\end{example}

\begin{figure}[htb!]
	\centering
	\includegraphics[width=0.7\textwidth]{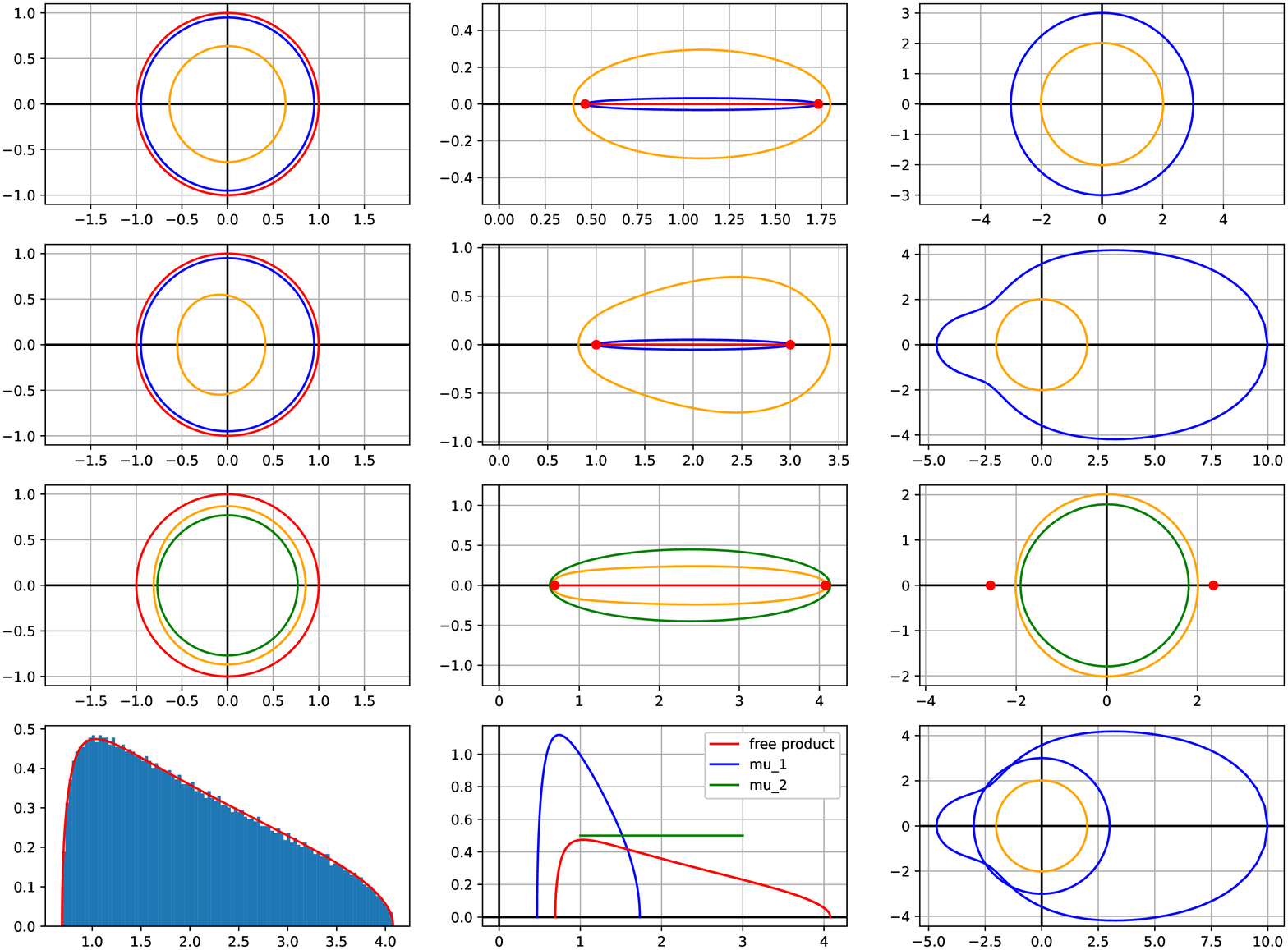}
	\caption{Free multiplication of Marchenko-Pastur and uniform distribution. See Example~\ref{ex:MPxunif}.}
	\label{fig:MPxunif}
\end{figure}

\begin{example}\label{ex:weirdexp}
	In Figure~\ref{fig:weird}, we consider the free convolution of the semicircle distribution with the distribution $\mu_2$ that has support in $[-\sqrt 3, \sqrt 3]$ and has a density
	\begin{equation*}
		\mathrm{d}\mu_2(x) = \frac{5\sqrt 3}{144} (x^2+1)^2.
	\end{equation*}
	This has the property that $G_{\mu_2}'(i) = G_{\mu_2}'(-i) = 0$ (and it does not have sqrt-behavior at the boundary), therefore its Cauchy transform is not invertible. The image via $\mathcal G_{\mu_2}$ of the blue circle of radius $r_A$ is a curve with two self-intersections. This is not a problem for the application of the Cauchy integral formula, as long as we consider a point that is inside the curve. The disadvantage of this example is that we need $r_B$ -- the radius of the orange circle -- to be quite small, and this means that we need to truncate the power series corresponding to $\mathcal G_{\mu_1 \boxplus \mu_2}$ to the first few terms. The numerical approximation that we get closely agrees with the histogram of the eigenvalues of the sum of a symmetric random matrix and a random matrix whose eigenvalues are samples from the weird distribution. We used $4000$ quadrature points for the discretization of the integral of the Cauchy transform of $\mu$ in this example.
\end{example}

\begin{figure}[htb!]
		\centering
		\includegraphics[width=0.7\textwidth]{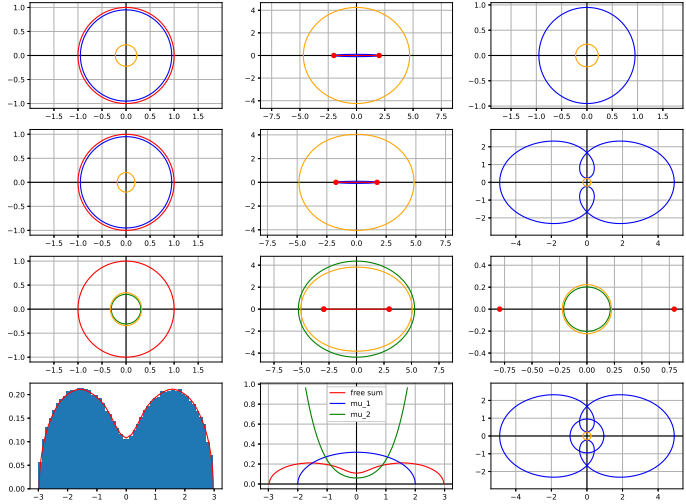}
		\caption{Free sum of semicircle law and the distribution from Example~\ref{ex:weirdexp}, with $m = 20$.}
		\label{fig:weird}
	\end{figure}

Although the theory does not extend to this case, we can also consider measures that have atoms. For a point measure, it is easy to explicitly write its Cauchy transform (which is not invertible when there are more than two atoms). In our code, we use an explicit expression for the part of the Cauchy transform corresponding to the atoms, and quadrature for the rest. When considering the free convolution $\mu$ of a measure with sqrt-behavior at the boundary with a discrete measure, it can happen that the support of $\mu$ splits into two or more distinct pieces; when this is \emph{not} the case, our algorithm can successfully approximate the density of $\mu$, as in the next example.

\begin{example}\label{ex:discrete}
	In Figure~\ref{fig:semicirclexdiscrete}, we consider the free multiplicative convolution of a shifted semicircle distribution with a discrete measure that has masses $\frac{1}{7}$ in the points $1, \frac{3}{2}, 2, \frac{5}{2}, 3, \frac{7}{2}, 4$. In this case, the T-transform is easy to compute directly, so we do not use the trapezoidal quadrature rule for this part of the algorithm.
\end{example}
\begin{figure}[htb!]
	\centering
	\includegraphics[width=0.7\textwidth]{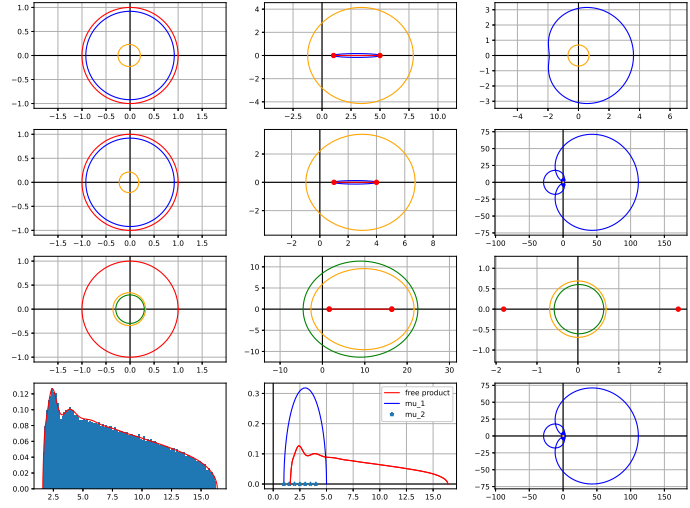}
	\caption{Free multiplication of shifted semicircle distribution with a discrete distribution. See Example~\ref{ex:discrete}.}
	\label{fig:semicirclexdiscrete}
\end{figure}

We comment on a warning on the choice of the parameters ($r_A$ and $\varepsilon$) in our algorithms. On the one hand, when $r_A$ is (relatively) far from $1$ and $\varepsilon$ is (relatively) large, the trapezoidal rules for the computation of $g(z)$ on the two intervals  $\left (\max\{\mathcal G_{\mu_1}(-r_A+\varepsilon), \mathcal G_{\mu_2}(-r_A+\varepsilon)\}, 0\right )$ and $\left (0, \min\{\mathcal G_{\mu_1}(r_A-\varepsilon), \mathcal G_{\mu_2}(r_A-\varepsilon)\}\right )$ converge very fast, allowing for a small value of quadrature points $N$. On the other hand, Theorem~\ref{thm:support} ensures the existence of a zero of $g'$ in the \emph{larger} intervals $(\max(G_{\mu_1}(a_1), G_{\mu_2}(a_2)), 0)$ and $(0, \min(G_{\mu_1}(b_1), G_{\mu_2}(b_2)))$. Therefore, an easy choice of $r_A$ and $\varepsilon$ may mean that we cannot successfully find the support of $\mu$. In our implementation, we give a warning when this happens, and then proceed to choose $\xi_a = \max\{\mathcal G_{\mu_1}(-r_A+\varepsilon), \mathcal G_{\mu_2}(-r_A+\varepsilon)\}$ or $\xi_b = \min\{\mathcal G_{\mu_1}(r_A-\varepsilon), \mathcal G_{\mu_2}(r_A-\varepsilon)\}$, which still gives visually good results, see the next example.

\begin{example}\label{ex:longunif}
    We consider the free additive convolution of a semicircle distribution with the uniform distribution on the interval $[-10, 10]$. We can compute the support of the result analytically and it is $[-10.3497, 10.3497]$. The choice of $\varepsilon = 0.05$ is not small enough to allow us to find the zeros of the derivative of $g(z)$ as defined in~\eqref{eq:inverse}. Choosing a smaller value $\varepsilon = 0.02$ results in an accurate approximation; see Figure~\ref{fig:smallepsilon}.
\end{example}
\begin{figure}[htb!]
	\centering
	\includegraphics[width=0.7\textwidth]{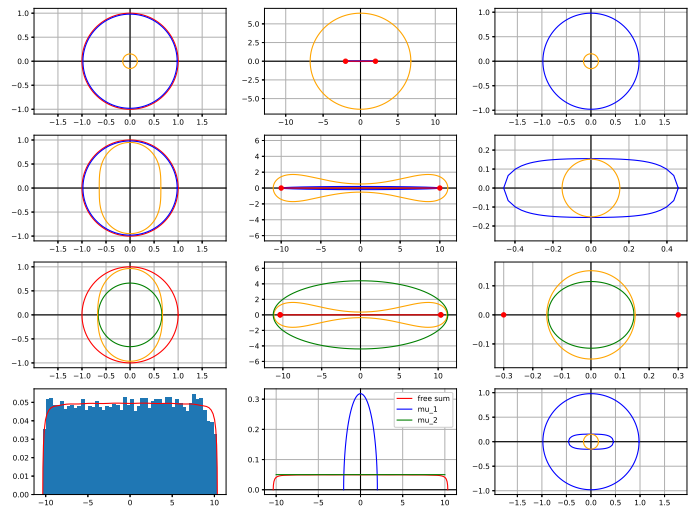}
	\caption{Free addition of semicircle distribution with uniform distribution on $[-20, 20]$. See Example~\ref{ex:longunif}.}
	\label{fig:smallepsilon}
\end{figure}

%=============================================================================================
\section{Conclusions}\label{sec:conclusions}
In this paper, we have proposed an algorithm that allows us to approximately compute the density of a measure with an sqrt-behavior at the boundary that results from the free (additive or multiplicative) convolution of two measures with compact support. Our methods use numerical quadrature, and the accuracy improves quickly when increasing the number of quadrature points.

We mention three possible directions in which our work could be extended. First, one could find more general assumptions under which the strategy of Theorem~\ref{thm:support} allows us to compute the support of $\mu$, or prove Theorem~\ref{thm:supporttimes} under the less restrictive assumptions of Theorem~\ref{thm:support}. Indeed, the numerical experiments seem to indicate that Algorithm~\ref{alg:multiplication} also works under the more general assumption that one of the input measures has sqrt-behavior at the boundary and the other one is a more general Jacobi measure. Second, it would be nice to extend the method to the case where the support is on more than one interval, for example, when considering the free convolution of a semicircle or Marchenko-Pastur distribution with a discrete measure. Finally, it would be interesting to see whether this approach can be extended to operator-valued free probability theory (see, e.g.,~\cite{Mai2018}), as this would allow for an extension of the algorithm for free additive convolution to an algorithm for the computation of rational functions in free random variables.

\paragraph{Acknowledgements.} The authors would like to thank Emmanuel Candès, Jorge Garza Vargas, and Iain Johnstone for inspiring discussions on this work. 

%The work of L.Y. is partially supported by the National Science Foundation under awards DMS-2011699 and DMS-2208163.

\appendix
	
\section{Proofs of the results in Section~\ref{sec:freeconvolution}}\label{sec:proofs}

	\begin{proof}[Proof of Theorem~\ref{thm:Ganalytic}]
		Given $z \in \mathbb{C} \backslash [a,b]$ we are going to prove that $G$ admits a series expansion in a neighborhood of $z$. Let us consider points $w \in \mathbb{C}$ such that
		\begin{equation}\label{eq:nbd}
			|z-w| < \frac{1}{2}\mathrm{dist}(z, [a, b]).
		\end{equation}
		Then, for all $z \in [a,b]$ we have $|w-z|/|x-z| < \frac{1}{2}$. Therefore, the series $\sum_{n \ge 0} \left ( \frac{w-z}{x-z}\right )^n$ converges uniformly to $\frac{x-z}{x-w}$ in the neighborhood~\eqref{eq:nbd}. By rearranging the terms, we get that
		\begin{equation*}
			(w-x)^{-1} = - \sum_{n \ge 0} (x-z)^{-n-1}(w-z)^n.
		\end{equation*}
		Therefore, by Fubini-Tonelli's theorem, we can exchange sum and integral in the definition of the Cauchy transform and we get
		\begin{equation*}
			G(z) = \sum_{n \ge 0} \left [ -\int_{a}^b (x-z)^{-(n+1)} \mathrm{d}\mu(x) \right ] (w-z)^n
		\end{equation*}
		for all $w$ in the neighborhood~\eqref{eq:nbd}. This proves that $G$ is analytic on $\mathbb{C} \backslash [a,b]$.
		
		To prove that $G$ is analytic at $\infty$, we need to show that $g(z) := G\left ( \frac{1}{z} \right )$ is analytic in zero. We have that
		\begin{equation*}
			g(z) = \int_a^b \frac{z}{1-zx} \mathrm{d}\mu(x).
		\end{equation*}
		The series $z \cdot \sum_{n \ge 0} (zx)^n$ converges uniformly to $\frac{z}{1-zx}$ for $|z| < \frac{1}{2} \min\left \{ \frac{1}{|a|}, \frac{1}{|b|} \right \}$ (with the convention that $\frac{1}{0} = +\infty$). Therefore, we can exchange sum and integral and we get
		\begin{equation*}
			g(z) = \sum_{n\ge 1} \left ( \int_a^b x^{n-1} \mathrm{d}\mu(x) \right ) z^n = z + \sum_{n\ge 2} \left ( \int_a^b x^{n-1} \mathrm{d}\mu(x) \right ) z^n
		\end{equation*}
		in the region where $|z| < \frac{1}{2} \min\left \{ \frac{1}{|a|}, \frac{1}{|b|} \right \}$, which proves that $G$ is also analytic in $z = \infty$.
	\end{proof}

The proof of Corollary~\ref{cor:limit} follows directly from the following theorem, which is a slight generalization of Theorem~\ref{thm:stiltjies}.
\begin{theorem}\label{thm:limit}
	Assume $\mathrm{d}\mu(x) = f(x)\mathrm{d}x$ for a continuous function $f$. For any $c < d \in (a, b)$ we have that
	\begin{equation*}
		-\frac{1}{\pi} \lim_{z \to 0 \atop z \in \mathbb{C}^+} \int_c^d \mathrm{Im}\left ( G(t + z) \right ) \mathrm{d}t = \mu([c, d]).
	\end{equation*}
\end{theorem}
\begin{proof}
	We denote $z = \varepsilon + iy$, with $\varepsilon \in \mathbb{R}$ and $y \in \mathbb{R}^+$. We have that
	\begin{equation*}
		\mathrm{Im}(G(t+z)) = \int_a^b \mathrm{Im} \left ( \frac{1}{t+\varepsilon - x + iy} \right ) \mathrm{d}\mu(x) = \int_a^b \frac{-y}{(t+\varepsilon-x)^2+y^2} \mathrm{d}\mu(x).
	\end{equation*}
	Thanks to Fubini-Tonelli's theorem, we can write
	\begin{align*}
		\int_c^d \mathrm{Im}(G(t+z)) \mathrm{d}t & = \int_a^b \left ( \int_c^d - \frac{y}{(t+\varepsilon-x)^2 + y^2} \mathrm{d}t \right ) \mathrm{d}\mu(x) \\
		& = - \int_a^b \left ( \int_{(c+\varepsilon-x)/y}^{(d+\varepsilon-x)/y} \frac{1}{1+{\mathcal t}^2} \mathrm{d}\mathcal t\right ) \mathrm{d}\mu(x) \\
		& = -\int_a^b \left ( \arctan \left ( \frac{d+\varepsilon-x}{y}\right ) - \arctan \left ( \frac{c+\varepsilon-x}{y}\right ) \right ) \mathrm{d}\mu(x),
	\end{align*}
	where we made the change of variables $\mathcal t = \frac{t+\varepsilon-x}{y}$. Now consider the function $F((\varepsilon,y), x) := \arctan \left ( \frac{d+\varepsilon-x}{y}\right ) - \arctan \left ( \frac{c+\varepsilon-x}{y}\right )$. We have that $|F((\varepsilon,y), x)| \le 2\pi$ and that $F((\varepsilon,y), x) \to_{(\varepsilon,y)\to(0,0^+)} \mathcal F(x)$ where
	\begin{equation*}
		\mathcal F(x) = \begin{cases}
			0 & \text{if } x \notin [c,d];\\
			\pi & \text{if } x \in (c,d),
		\end{cases}
	\end{equation*}
	for all $x$ except for $x \in \{c, d\}$. The function $\mathcal F$ is measurable, and it agrees $\mu$-almost everywhere with the $\mu$-almost everywhere existing pointwise limit. Hence we can apply the dominated convergence theorem. This, together with the fact that $\mathrm{d}\mu(\{c,d\}) = 0$, implies
	\begin{equation*}
		\lim_{(\varepsilon, y) \to (0, 0^+)} \int_c^d \mathrm{Im}(G(t + \varepsilon + iy)) \mathrm{d}t = - \int_a^b \mathcal F(x) \mathrm{d}\mu(x) = -\pi \mu([c, d]).\qedhere
	\end{equation*}
\end{proof}

\begin{proof}[Proof of Theorem~\ref{thm:Ranalytic}]
Theorem 17 in Chapter 3 of~\cite{Mingo2017} states that if the support is contained in an interval $[-r, r]$, then $R(z)$ is analytic in a disk of radius $\frac{1}{6r}$. Now recall that the function $G(z)$ is analytic on $\mathbb{C} \backslash [a, b]$; by the inverse function theorem for holomorphic functions (Theorem~\ref{thm:inversefunction}), it is invertible in a neighborhood of a point $z \in \mathbb{C} \backslash[a,b]$ whenever $G'(z) \neq 0$. We have that
\begin{equation*}
	G'(z) = - \int_a^b \frac{1}{(x-z)^2} \mathrm{d}\mu(x), 
\end{equation*}
which is always nonzero for $z \in \mathbb{R} \backslash [a,b]$, from which we can conclude the second part of the theorem.
\end{proof}

\section{Proof of Theorem~\ref{thm:supporttimes}}

% Assumptions needed in the proof:
% \begin{itemize}
%     \item $t(z) := \frac{z}{1+z} T_{\mu_1}^{-1}(z) T_{\mu_2}^{-1}(z)$ needs to be well defined (which also means that $T_{\mu_1}^{-1}$ and $T_{\mu_2}^{-1}$ are single-valued).
%     \item $T_{\mu_1}(\mathbb{C}^{-}) \cap T_{\mu_2}(\mathbb{C}^{-})$ needs to be bounded (which happens for instance when one of the two measures has sqrt-decay at the boundary, or is Jacobi with $\alpha, \beta > 0$.
%     \item $\mu_1$ and $\mu_2$ have support on the positive real axis.
% \end{itemize}

\begin{theorem}[{Inverse function theorem for holomorphic functions~\cite[Theorem 7.5]{Fritzsche2002}}]\label{thm:inversefunction}
Let $U$ be an open set in $\mathbb{C}$, let $f: U \to \mathbb{C}$ be holomorphic, and let $p \in U$ be such that $f'(p) \neq 0$. Then there exists an open neighborhood $V$ of $p$ such that $f: V \to f(V)$ is a biholomorphism.
\end{theorem}
Consequence: Under the assumptions of Theorem~\ref{thm:inversefunction}, if $f(p) \in \mathbb{R}$ then there exists precisely one curve in $V$ passing through $p$ on which $f$ is real-valued.

\begin{theorem}[Complex Morse lemma~\cite{Zoladek2006}]\label{thm:morse}
Let $U$ be an open set in $\mathbb{C}$, let $f: U \to \mathbb{C}$ be a holomorphic function, and let $p\in U$ be such that $f'(p) = 0$ and $f''(p) \neq 0$ (a non-degenerate critical point of $f$). Then there exist neighborhoods $\tilde U$ of $p$ and $V$ of $0$, and a bijective holomorphic function $\varphi : V \to \tilde U$ such that $\varphi(0) = p$ and $f(\varphi(w)) = f(p) + w^2$.
\end{theorem}
Consequence: Under the assumptions of Theorem~\ref{thm:morse}, if additionally $f(p) \in \mathbb{R}$, there exist \emph{exactly two} curves passing through $p$ on which $f$ is real-valued. These curves are the images via $\varphi$ of the intersection of $V$ with the x-axis and y-axis.

\begin{lemma}\label{lem:expansionT}
Let $\mu$ be a measure with sqrt-behavior at the boundary. For $z$ in a neighborhood of $0$ we have
\begin{equation*}
    T_{\mu}^{-1}(z) = \frac{1}{z} \mathbb{E}[X_{\mu}] + \text{analytic function}.
\end{equation*}
\end{lemma}
\begin{proof}
    We have that
    \begin{align*}
        T_{\mu}\left ( \frac{1}{z} \right ) & = \int_a^b \frac{x}{1/z-x} \mathrm{d}\mu(x) = \int_a^b \frac{zx}{1-zx}\mathrm{d}\mu(x)\\& = \int_a^b \left ( \sum_{n=1}^{+\infty} (zx)^n \mathrm{d}\mu(x)\right ) = \sum_{n=1}^{+\infty} z^n \int_a^b x^n \mathrm{d}\mu(x) = z \mathbb{E}[X_\mu] + z^2 f(z),
    \end{align*}
    for some holomorphic function $f(z)$. The fourth inequality follows Fubini-Tonelli's theorem because the series converges absolutely in a neighborhood of zero. Let us apply the inverse function theorem (Theorem~\ref{thm:inversefunction}) to $\varphi(z) := T_{\mu}\left ( \frac{1}{z} \right )$ in a neighborhood of $0$, where $\varphi(0) = 0$ and $\varphi'(0) = \mathbb{E}[X_{\mu}] > 0$. We get that $\varphi^{-1}(z) = \frac{z}{\mathbb{E}[X_{\mu}]} + \tilde \varphi(z)$
    for a holomorphic $\varphi(z)$, therefore
    \begin{equation*}
        T_{\mu}^{-1}(z) = \frac{\mathbb{E}[X_{\mu}]}{z} \left ( \frac{1}{1 + z \tilde \varphi(z) \mathbb{E}[X_{\mu}]} \right ) = \frac{\mathbb{E}[X_{\mu}]}{z} + \frac{\tilde \varphi(z)}{1 + z\varphi(z) \mathbb{E}[X_{\mu}]},
    \end{equation*}
    where the function $\frac{\tilde \varphi(z)}{1 + z\varphi(z) \mathbb{E}[X_{\mu}]}$ is holomorphic in a neighborhood of zero.
\end{proof}

\begin{proposition}\label{prop:border}
Assume that $\mu$ satisfies the Assumptions~\ref{ass:prod}. Then $T_{\mu}([a,b])$ is a bounded continuous curve. 
%(possibly unbounded) continuous curve. The curve is bounded if $\mu$ has sqrt-behavior at the boundary or if it is Jacobi with $\alpha, \beta > 0$.
\end{proposition}
\begin{proof}
    There is a similar result for the additive case~\cite[Proposition A.8]{Olver2012}; Proposition~\ref{prop:border} is obtained by substituting $\psi(x)$ with $x\psi(x)$ in all the integrals in the proof of~\cite[Proposition A.8]{Olver2012}.
\end{proof}

\begin{proposition}\label{prop:expansion}
If $\mu$ is a measure with sqrt-behavior at the boundary and the support $[a,b]$, then $T_{\mu}^{-1}$ is analytic in a neighborhood of $T_{\mu}(a)$ and $T_\mu(b)$, with zero derivative and nonzero second derivative. 
% Jacobi measure and $\beta > 0$ then
% \begin{equation*}
%     T_{\mu}(z) = T_{\mu}(b) + C (z-b)^\beta + o(z-b)^\beta.
% \end{equation*}
% If $\mu$ is precisely a Jacobi measure (i.e. $\psi(a), \psi(b) \neq 0$ and $\beta > 0$ then $C \neq 0$ and
% \begin{equation*}
%     T_\mu^{-1}(w) = b + C(T_\mu(b) - w)^{1/\beta} + o(T_\mu(b) - w)^{1/\beta} .
% \end{equation*}
% If $\mu$ is precisely sqrt-behavior then $T_\mu^{-1}$ is analytic in a neighborhood of $T_\mu(b)$, with zero derivative and nonzero second derivative. Similar things near $T_\mu(a)$.
\end{proposition}
\begin{proof}
    Note that $T_{\mu}(z) = c G_{\nu}(z)$ for a suitable constant $c := \int_a^b x \mathrm{d}\mu(x) > 0$ and a suitable measure $\nu$ defined by $\mathrm{d}\nu(x) = \frac{1}{c} x \mathrm{d}\mu(x)$. If $\mu$ is a measure with sqrt-behavior at the boundary, so is $\nu$. Therefore, we can apply~\cite[Proposition A.6]{Olver2012} to $\nu$ and get the stated result.  
\end{proof}

We are now ready to prove Theorem~\ref{thm:supporttimes} by following similar steps to the proof of the additive case~\cite[Theorem 2.2]{Olver2012}.

\begin{proof}[Proof of Theorem~\ref{thm:supporttimes}]
The set $T_{\mu_1}(\mathbb{C}^{-}) \cap T_{\mu_2}(\mathbb{C}^{-})$ is bounded and has a continuous border by Proposition~\ref{prop:border}. Let $\nu$ be on the border, and assume without loss of generality that $\nu \in \partial T_{\mu_1}(\mathbb{C}^{-})$. This implies that $T_{\mu_1}^{-1}(\nu) \in \mathbb{R}$.

\paragraph{Claim \#1:} $\mathcal{I}t(\nu) > 0$.

\begin{quote}
    \textbf{Proof of claim \#1:} Since $T_{\mu_1}^{-1}(\nu) \in \mathbb{R}$ we need to show that $ \mathcal{I} \left ( \frac{\nu T_{\mu_2}^{-1}(\nu)}{1+\nu} \right ) > 0$. The function $f(y) := \frac{1}{y} + \frac{1}{y T_{\mu_2}(y)}$ is holomorphic for $y \in \mathbb{C}^{-}$ (this is true since $T_{\mu_2}$ is holomorphic and nonzero here), therefore the function $\tilde f(y) := \mathcal{I}(f(y))$ is harmonic. If $y \notin [a,b]$ we have $\tilde f(y) = 0$; if $y \in \mathbb{R} \backslash [a,b]$ we have, by the Stiltjes inversion theorem (Theorem~\ref{thm:stiltjies}),
    \begin{equation*}
        \lim_{\varepsilon \to 0^-} \mathcal{I}(T_{\mu_2}(y+\varepsilon I)) = \pi y \mathrm{d}\mu_2(y) > 0,
    \end{equation*}
    therefore $\tilde f(y) \le 0$. By the modulus maximization property of harmonic functions, we have that $\tilde f(y) < 0$ for all $y \in \mathbb{C}^{-}$. Substituting $y = T_{\mu_2}^{-1}(\nu)$ we get $\mathcal{I} \left ( \frac{1 + \nu}{\nu T_{\mu_2}^{-1}(\nu)} \right ) < 0$ and therefore $ \mathcal{I} \left ( \frac{\nu T_{\mu_2}^{-1}(\nu)}{1+\nu} \right ) > 0$, which proves the claim.
\end{quote}

By Lemma~\ref{lem:expansionT}, we have that, in a neighborhood of zero, 
\begin{equation}\label{eq:expansiont}
t(z) = \frac{z}{1+z} \frac{1}{z^2} \mathbb{E}[X_{\mu_1}]\mathbb{E}[X_{\mu_2}] + f(z) = \frac{C}{z} + f(z)
\end{equation}
for a positive real constant $C$ and an analytic function $f(z)$, therefore $\mathcal{I}(t(z)) < 0$ in a neighborhood of zero. It follows that there exists a point $\zeta \in T_{\mu_1}^{-1}(\mathbb{C}^{-}) \cap T_{\mu_2}^{-1}(\mathbb{C}^{-})$ in which $t(\zeta) \in \mathbb{R}$. By the consequence of Theorem~\ref{thm:inversefunction} there exists a curve $\Gamma$ in $\mathbb{C}^{-}$ on which $t$ is real-valued. This curve cannot intersect itself (otherwise, we would have a loop and, by analyticity, the function $t$ would be a constant), is bounded, and cannot touch the part of the boundary in the complex plane. Therefore, it will connect with the real axis; there needs to be exactly one point to the left of $0$ and one point to the right of $0$, which we will call $\xi_a$ and $\xi_b$, respectively. Indeed, if there were two points on the same side of zero, we would get a closed loop on which $t$ is real-valued in a region where $t$ is analytic, leading to $t$ being a constant. For the same reason, $\Gamma$ is the only curve inside $T_{\mu_1}^{-1}(\mathbb{C}^{-}) \cap T_{\mu_2}^{-1}(\mathbb{C}^{-})$ where $t$ is real-valued. We have proved that $\Gamma$ divides $T_{\mu_1}^{-1}(\mathbb{C}^{-}) \cap T_{\mu_2}^{-1}(\mathbb{C}^{-})$ into an interior region $D$ where $\mathcal{I}(t(z)) > 0$ and an exterior region such that $\mathcal{I}(t(z)) < 0$.

\paragraph{Claim \#2:} Let $\xi_b^{1} := T_{\mu_1}(b_1)$ and $\xi_b^{2} := T_{\mu_2}(b_2)$. Then $\xi_b < \min\{\xi_b^1, \xi_b^2\}$.

\paragraph{Claim \#3:} $t'(z) > 0$ in a left neighborhood of $\min\{\xi_b^1, \xi_b^2\}$ on the real axis (which includes the point  $\min\{\xi_b^1, \xi_b^2\}$).

\begin{quote}
    \textbf{Proof of claim \#3:} For $z \in (0,\min\{\xi_b^1, \xi_b^2\})$ we have that
    \begin{equation*}
        t'(z) = \frac{z}{1+z}(T_{\mu_1}^{-1})'(z) T_{\mu_2}^{-1}(z) + \frac{z}{1+z} T_{\mu_1}^{-1}(z) (T_{\mu_2}^{-1})'(z) + \frac{1}{(1+z)^2} T_{\mu_1}^{-1}(z) T_{\mu_2}^{-1}(z).
    \end{equation*}
    Without loss of generality, we can assume that $\xi_b^1 < \xi_b^2$. In this case $(T_{\mu_1}^{-1})'(\xi_b^1) = 0$  (thanks to Proposition~\ref{prop:expansion}), therefore we need to show that
        \begin{equation*}
        t'(\xi_b^1) = \frac{ T_{\mu_1}^{-1}(\xi_b^1)}{1 + \xi_b^1} \left ( \frac{\xi_b^1}{T_{\mu_2}'(T_{\mu_2}^{-1}(\xi_b^1))} + \frac{T_{\mu_2}^{-1}(\xi_b^1)}{1+\xi_b^1} \right ) > 0.
    \end{equation*}
    The first factor, $\frac{ T_{\mu_1}^{-1}(\xi_b^1)}{1 + \xi_b^1}$, is positive. Let us look at the second factor. Letting $w = T_{\mu_2}^{-1}(\xi_b^1)$, this is equivalent to show that
    \begin{equation}\label{eq:musthold}
        \frac{T_{\mu_2}(w)}{T_{\mu_2}'(w)} + \frac{w}{1 + T_{\mu_2}(w)} = \frac{T_{\mu_2}(w) + T_{\mu_2}^2(w) + wT_{\mu_2}'(w)}{(1 + T_{\mu_2}(w)) T_{\mu_2}'(w)} > 0.
    \end{equation}
    \begin{itemize}
        \item We have $1 + T_{\mu_2}(w) = 1 + \xi_b^1 > 0$.
        \item We have $T_{\mu_2}'(w) = - \int_{a_2}^{b_2} \frac{x}{(w-x)^2} \mathrm{d}x < 0$
        \item We have
        \begin{align*}
            T_{\mu_2}(w) &+ T_{\mu_2}^2(w) + wT_{\mu_2}'(w) \\
            & = \int_{a_2}^{b_2} \frac{x}{w-x} \mathrm{d}x + \left ( \int_{a_2}^{b_2} \frac{x}{w-x} \mathrm{d}x \right )^2 - \int_{a_2}^{b_2} \frac{xw}{(w-x)^2} \mathrm{d}x\\
            & < \int_{a_2}^{b_2} \frac{x}{w-x} \mathrm{d}x + \int_{a_2}^{b_2} \frac{x^2}{(w-x)^2} \mathrm{d}x - \int_{a_2}^{b_2} \frac{xw}{(w-x)^2} \mathrm{d}x = 0,
        \end{align*}
        where we used Jensen's inequality for the convex function $x \mapsto \frac{x}{w-x}$.
    \end{itemize}
    Hence,~\eqref{eq:musthold} holds, proving the claim.

\end{quote}

\begin{quote}
    \textbf{Proof of claim \#2:} The function $t$ is analytic in a neighborhood of $\xi_b$. We have that $t'(z)$ goes to $-\infty$ in a (right) neighborhood of zero (because of~\eqref{eq:expansiont}), and $t'(z)$ is positive in a (left) neighborhood of $\min\{\xi_b^1, \xi_b^2\}$ because of Claim \#3. Therefore, there exists at least one point in the interval $(0, \xi_b)$ where $t'$ is zero. In each of these points, $p$ the consequence of Morse Lemma (Theorem~\ref{thm:morse}) tells us that there are two curves emanating from $p$ on which $t$ is real-valued; however, the only real-valued curves are the x-axis and $\Gamma$. Therefore there can be only one such point, and it has to be $\xi_b$. 
\end{quote}

With a similar argument, defining $\xi_a^{1} := T_{\mu_1}(a_1)$ and $\xi_a^{2} := T_{\mu_2}(a_2)$, one can show that $\xi_a > \max\{\xi_a^1, \xi_a^2\}$. 

Now, $t(z) = T_{\mu}^{-1}(z)$ in a neighborhood of $0$, and $t(z)$ is a single-valued function in $D$, therefore $t(z) = T_{\mu}^{-1}(z)$ inside all of $D$ (this step is because single-valuedness implies that two functions which coincide in an open set coincide everywhere). Since the only real-valued curve contained in $T_\mu(\mathbb{C}^{-})$ is the boundary, $\Gamma$ is the boundary of $T_\mu(\mathbb{C}^{-})$.

We want to use the inverse function theorem on $t$ inside $\bar D$ (the closure), except for the points $\xi_a$ and $\xi_b$. We check that the derivative of $t(z)$ is never zero on this set: we have $t'(z) = \frac{1}{T_{\mu}'(T_{\mu}^{-1}(z))}$ when $z \in D$ (open) and $\mathcal{I}(z) \neq 0$; moreover if $t'(z)$ had other zeros in the segment $(\xi_a, \xi_b)$ or on $\Gamma$ we would have another real-valued curve and this is impossible. The inverse function theorem (Theorem~\ref{thm:inversefunction}) implies that $t^{-1}$ exists and is analytic everywhere in $t(D)$, on $(a,b)$ where $a := t(\xi_a)$ and $ b := t(\xi_b)$, and on $\Gamma$ (except for, in general, the points $a$ and $b$). Note that we can conclude that $t^{-1}$ is also analytic on $(a,b)$ because $t$ is defined, analytic, and invertible also on the segments $(\xi_a, 0)$ and $(0, \xi_b)$ and in a domain $D'$ which is symmetric of $D$ with respect to the real axis.

Let us look at the density function $f(x)$ of $\mu$. We have that 
\begin{equation*}
x f(x) = \frac{1}{\pi} \lim_{\varepsilon \to 0^+}\mathcal{I} (t^{-1}(x - i\varepsilon)) \text{ for } x \in (a,b),
\end{equation*}
therefore $f(x)$ is analytic in $(a,b)$. Let us look at point $a$ (the argument for $b$ is entirely analogous). There are two real-valued curves passing from the point $\xi_a$, therefore the function $t$ has the form $$t(z) = a + c_2(z-\xi_a)^2 + (z-\xi_a)^3h(z)$$ in a neighborhood of $\xi_a$, for some analytic function $h(z)$, and similarly for $\xi_b$. %\textcolor{red}{[Actually this is the converse of the Morse Lemma, does it hold as well? If it had a $c_1 \neq 0$, it would be a local diffeomorphism with a neighborhood of $0$ in $\mathbb{C}$, and it would only have one real curve; if $c_2 =0$ could we way that there would be THREE real curves at least? If $c_1 = \ldots = c_{k-1} = 0, c_k \neq 0$ we would have $t(z) = a + c_k(z-\xi_a)^k s(z)$ for some holomorphic $s(z)$ in a neighborhood of $\xi_a$. ]}
By Theorem~\ref{thm:morse}, there exists a local biholomorphism $\varphi$ such that $\varphi(0) = \xi_a$ and\footnote{The minus sign arises because we want to choose $\varphi$ in such a way that a right neighborhood of $\xi_a$ on the real axis is sent into a left neighborhood of $a$ on the real axis, and a piece of $\Gamma$ close to $\xi_a$ is sent to a right neighborhood of $a$ on the real axis.} $t(\varphi(z)) = a - z^2$.
%\textcolor{red}{(if it is $a - z^2$ then it works. How do I "choose" the sign?)}.
Therefore, by setting $w = a- z^2 $ we get the expression
\begin{equation*}
    t^{-1}(w) = T_{\mu}(w) = \varphi\left (\sqrt{a-w}\right ) = c_0 + c_{1/2}\sqrt{a-w} + c_1(a-w) + c_{3/2}(a-w)^{3/2} + \ldots,
\end{equation*}
where all the coefficients are real because $T_{\mu}(w)$ is real for $w \in \mathbb{R} \backslash[a,b]$. Therefore, for $x \in (a,b)$ sufficiently close to $a$ we have
\begin{align*}
    x f(x) & = \frac{1}{\pi} \lim_{\varepsilon \to 0^+} \mathcal{I}(T_\mu(x - i\varepsilon)) \\
    & = \frac{1}{\pi} \lim_{\varepsilon \to 0^+} \mathcal{I}( \sqrt{a - x+i\varepsilon }(c_{1/2} + c_{3/2}(a-w) + c_{5/2}(a-w)^2 + \ldots))\\
    & = \frac{1}{\pi} \lim_{\varepsilon \to 0^+} \sqrt{x-a} \cdot h(x),
\end{align*}
for a function $h(x)$ that is analytic in a right neighborhood of $a$. The same argument holds in a neighborhood of $b$.
Therefore, $\mu$ is a measure with sqrt-behavior at the boundary, that is, with the form $\mathrm{d}\mu(x) = \sqrt{x-a} \sqrt{b-x} \psi(x)$ for an analytic function $\psi(x)$, with an invertible T-transform. This proves Theorem~\ref{thm:supporttimes} and allows us to reconstruct the measure using the truncated series expansion.
\end{proof}

% \textcolor{blue}{Note: It would be nice to assume that $\mu_1$ has sqrt-behavior and that $\mu_2$ is a Jacobi measure. For now, we are not sure about the proof of Theorem~\ref{thm:supporttimes} in this case, though.}

% A consequence of the proof of Theorem~\ref{thm:supporttimes} is that if the assumptions are not satisfied, but $t(z)$ is well defined, and its derivative has two zeros (one on the left and one on the right of zero) then these zeros defines the support, and the resulting measure has sqrt-behavior at the boundary.
	
\bibliographystyle{abbrv}
\bibliography{bib}

\end{document}